\documentclass[10pt,a4paper]{article}

\usepackage[theorem/counter/prefix = subsection]{mac-bundle}
\usepackage{flags}
\usepackage{bookmark} 
\usepackage{geometry}[margin=0cm]

\DeclareTheorem[plain]{question}{Question}

\newwrapper {macros}
\newwrapper {citetree}

\input{sketch.macro}
\newwrapper{paths}
\begin{paths}
  \makeatletter
  
  \newcommand \pathdef [3] {%
    \expandafter\def\csname @paths@#1\endcsname{\cite[#3]{#2}}%
  }
  
  \newcommand \pathcite [1] {%
    \ifcsdef{@paths@#1}{%
      \csname @paths@#1\endcsname%
    }{%
      \PackageWarning{paths}{Path "#1" not found}%
    }%
  }
  \makeatother
\end{paths}

\pathdef {cis/fibration/right/is-smooth} {Cisinski:HigherCats} {None}
\pathdef {cis/fibre-product/in-slice} {Cisinski:HigherCats} {Prop.~6.6.7}
\pathdef {cis/adjoint/criterion/slice} {Cisinski:HigherCats} {Prop.~6.1.11}
\pathdef {cis/localization/of-slice} {Cisinski:HigherCats} {Cor.~7.6.13}
\pathdef {htt/cofinal/definition} {HTT} {Def.~4.1.1.1}
\pathdef {cis/model/functor-category/is-category} {Cisinski:HigherCats} {Cor.~3.2.10}
\pathdef {cis/presheaf/small/definition} {cisinski2019higher} {Def.~6.1.18}
\pathdef {cis/kan-extension/base-change/smooth-right} {cisinski2019higher} {Thm.~4.4.26}
\pathdef {htt/limit/of-functors} {HTT} {Cor.~5.1.2.3}
\pathdef {hag/context/definition} {toen2008homotopical} {Def.~1.3.2.13}
\pathdef {universe/axiom/inaccessible-cardinal} {williams1969grothendieck} {}
\pathdef {hag/stack/n-geometric/definition} {toen2008homotopical} {Def.~1.3.3.1}
\pathdef {htt/localization/left-exact/criterion} {HTT} {Prop.~6.2.1.1}
\pathdef {cis/presheaf/colimit-of-yoneda-is-final} {cisinski2019higher} {Prop.~6.2.13}
\pathdef {cis/colimit/definition} {cisinski2019higher} {Def.~6.2.2}
\pathdef {htt/factorization-system/definition} {HTT} {Def.~5.2.8.8}
\pathdef {ha/twisted-arrow/is-classified-by-yoneda} {HA} {Prop.~5.2.1.11}
\pathdef {cis/universe/homotopy-hypothesis} {cisinski2019higher} {Thm.~7.8.9}
\pathdef {cis/functions-fibrations/integral/essentially-surjective} {cisinski2019higher} {Cor.~5.4.10}
\pathdef {cis/fibration/right/slice-criterion} {cisinski2019higher} {Prop.~4.1.2}
\pathdef {cis/functions-fibrations/integral/is-fully-faithful} {cisinski2019higher} {Cor.~5.4.7}
\pathdef {cis/fibration/universal/definition} {cisinski2019higher} {Def.~5.2.3}
\pathdef {cis/homotopy-colimit/is-colimit} {cisinski2019higher} {Remark 7.9.10}
\pathdef {cis/yoneda-lemma} {cisinski2019higher} {Thm.~5.8.13}
\pathdef {htt/model/homotopy-limit/is-limit} {HTT} {Thm.~4.2.4.1}
\pathdef {htt/cone/lifting/unique} {HTT} {Prop.~2.1.2.5}
\pathdef {htt/cone/transport} {HTT} {Prop.~2.1.2.1}
\pathdef {htt/topos/descent} {HTT} {Thm.~6.1.3.9}
\pathdef {htt/hypercover/n-coskeletal/is-effective} {HTT} {Lemma 6.5.3.9}
\pathdef {htt/hypercover/hypercomplete/as-localization} {HTT} {Cor.~6.5.3.13}
\pathdef {hag/stack/n-geometric/as-quotient} {toen2008homotopical} {Prop.~1.3.4.2}
\pathdef {htt/yoneda/preserves-limits} {HTT} {Prop.~5.1.3.2}
\pathdef {hag/topos/segal-groupoid/definition} {toen2008homotopical} {Def.~1.3.1.6}
\pathdef {hag/stack/deligne-mumford/definition} {toen2008homotopical} {Def.~2.1.1.4}
\pathdef {htt/bousfield-localization/recognition/local-object} {HTT} {Prop.~5.5.4.2.(1)}
\pathdef {htt/comma/source-map-is-cartesian} {HTT} {Cor.~2.4.7.12}
\pathdef {htt/idempotent/walking} {HTT} {Def.~4.4.5.2}
\pathdef {htt/compact/generation/from-localization} {HTT} {Cor.~5.5.7.3}
\pathdef {htt/bousfield-localization/existence} {HTT} {Prop.~5.5.4.15}
\pathdef {cis/bousfield-localization/existence} {cisinski2019higher} {Prop.~7.11.4}
\pathdef {cis/presheaf/map-out} {cisinski2019higher} {Thm.~6.3.13}
\pathdef {htt/presheaf/map-out} {HTT} {Thm.~5.1.5.6}
\pathdef {htt/bousfield-localization/map-out} {HTT} {Prop.~5.5.4.20}
\pathdef {htt/colimit-sketch/localization} {HTT} {Prop.~5.6.3.2}
\pathdef {htt/presentable/definition} {HTT} {311}
\pathdef {cis/presentable/definition} {cisinski2019higher} {7.11.5}
\pathdef {htt/sifted/definition} {HTT} {Def.~5.5.8.1}
\pathdef {htt/presheaf/map-out/right-adjoint/is-pullback} {HTT} {Prop.~5.2.6.3}
\pathdef {htt/sifted/colimits-preserve-products} {HTT} {Lemma 5.5.8.11}
\pathdef {htt/filtered/criterion/left-exact} {HTT} {Prop.~5.3.3.3}
\pathdef {htt/presentable/criterion} {HTT} {Thm.~5.5.1.1}
\pathdef {htt/presheaf/map-out/extension/criterion} {HTT} {Lemma 5.1.5.5}
\pathdef {htt/kan-ext/definition} {HTT} {Def.~4.3.3.2}
\pathdef {htt/kan-ext/definition/fully-faithful} {HTT} {Def.~4.3.2.2}
\pathdef {cis/kan-ext/local-criterion} {cisinski2019higher} {Prop.~6.4.8}
\pathdef {htt/universe/locally-small/definition} {HTT} {below Prop.~5.4.1.7}
\pathdef {htt/universe} {HTT} {\S1.2.15}
\pathdef {htt/presentable/representable-iff-continuous} {HTT} {Prop.~5.5.2.2}
\pathdef {htt/slice/definition/joyal} {HTT} {Prop.~1.2.9.2}
\pathdef {htt/universe/locally-small/exponential} {HTT} {Ex.~5.4.1.8}
\pathdef {cis/localization/bousfield/definition} {cisinski2019higher} {Def.~7.11.1}
\pathdef {cis/localization/bousfield/criterion} {cisinski2019higher} {Prop.~7.11.2}
\pathdef {htt/compact/presheaf/criterion} {HTT} {Prop.~5.3.4.17}
\pathdef {htt/presentable/prl-prr} {HTT} {Cor.~5.5.3.4}

\title{Adjoining colimits}
\author{Andrew W. Macpherson} 
\begin{document}

\maketitle

\begin{abstract}

  This paper develops a theory of colimit sketches `with constructions' in higher category theory, formalising the input to the ubiquitous procedure of adjoining specified `constructible' colimits to a category such that specified `relation' colimits are enforced (or preserved).
  From a more technical standpoint, sketches are a way to describe dense functors using techniques from the homotopy theory of diagrams.
  
  We establish basic properties of diagrams in an $\infty$-category $C$ as a model for presheaves on $C$ and Bousfield localisations thereof, discuss extensions of functors and adjunctions, and equivalences of sets of diagrams.
  We introduce categories of presheaves which are `constructible in one step' by a set of diagrams and explore, via well-known examples, when constructible cocompletion is idempotent, i.e.~when any \emph{iterated} construction can be completed in one step.
  
\end{abstract}

\tableofcontents
\section{Introduction}

Every construction in higher category theory has a universal property. Typically, a universal property has the following form:
\begin{quote}

  \emph{$C'$ is the universal object under $C$ that admits \texttt{<construction>} such that \texttt{<relation>} is enforced.}\footnote{Or `preserved' if \texttt{<relation>} is already true in $C$.}
  
\end{quote}
This paper sets out a language of \emph{colimit sketches with constructions} which formalises this statement in the case that $C$ and $C'$ are $\infty$-categories and both \emph{relations} and \emph{constructions} can be expressed using colimits.
We describe how to exhibit $C'$ as a category of \emph{modules} for a sketch on $C$ and discuss its universal property and methods for `diagrammatic' manipulation of its objects.
Our language makes explicit the r\^ole colimit diagrams play as \emph{presentations} of objects of $\infty$-categories, expressing how they are put together from objects of a dense subcategory.  

\paragraph{What is this good for?}
The language developed in this paper may be useful to theory builders embarking on a construction of their own `designer' $\infty$-category.
Colimit sketches in higher category theory encompass numerous fundamental examples:
\begin{itemize}
  
  \item
    Non-Abelian derived categories \cite[\S5.5.8]{HTT} are universal among categories that admit \emph{sifted colimits} (and not preserving any specified relations). Alternatively, they are universal among categories admitting \emph{all colimits} and preserving \emph{finite coproducts}.
    
  \item
    The category of spaces with hypoabelian fundamental group is universal among categories under $\point$ that admit \emph{all colimits} and \emph{contract acyclic spaces} \cite{farjoun2006cellular}. (This example is culturally significant because the associated localisation functor is the Quillen plus construction.)
    
  \item 
    Morel-Voevodsky's category of motivic spaces, resp.~Robalo's category of non-commutative motives are universal among categories under $\Scheme$, resp.~$\mathbf{nc}\Scheme$, admitting \emph{all colimits} such that \emph{Nisnevich descent} is preserved and \emph{$\mathbb{A}^1$-localisation} is enforced \cite{robalo2012noncommutative}.\footnote{The Cisinski-Tabuada construction \cite{Blumberg_2013} admits a similar, but dual, description in terms of attaching limits.}
  
  \item 
    The category of $n$-geometric stacks is the universal category under the category of $(n-1)$-geometric stacks that admits \emph{descent along $(n-1)$-geometric Segal groupoids} such that \emph{\'etale descent} is preserved \cite[\S1.3.4]{toen2008homotopical}. 
      
\end{itemize}
New examples are cropping up all the time --- to name just a few, new applications of finite product theories \cite{carchedi2013theories}, new flavours of motives \cite{kahn2020motives}, new flavours of stack theory \cite{nuiten2018lie}.
In all cases, their universal properties are the starting point of many arguments about these objects. 

Even in situations where the ultimate object of interest is not a category but something simpler like a vector space or number, it is common to introduce $\infty$-categories defined by properties of, or very similar to, the preceding form as an auxiliary tool. 
Indeed, this is precisely the motivation for defining the preceding examples which are designed to study non-Abelian cohomology, homotopy groups of group completions, motivic cohomology or algebraic K-theory, and moduli problems with higher symmetries, respectively.

\subsection{What is done in this paper}

The main contributions of this paper are as follows:
\begin{itemize}
  
  \item 
    Introduce a homotopy theory of higher categorical diagrams which formalises the r\^ole of $\Cat\slice C$ as a model for $\P(C)$ and its localisations (\S\ref{diagrams/}) in a way that mirrors Grothendieck's development of the theory of \emph{localisateur fondamentals}.
    
  \item 
    Develop the language of sketches \emph{\`a la} Ehresmann; develop properties of dense functors and apply this to obtain extension theorems and universal properties of categories of modules over a sketch; provide a formula (`plus' construction) for the localisation functor (\S\S\ref{sketch/}, \ref{construction/}).
    
  \item 
    Various applications of the above to what may collectively be termed `classification' questions; discussion of Morita equivalence of theories (\S\ref{saturation/}) and localisations of $\Space$ (\S\ref{sketch/on-point/}); applications to iterated colimits (\S\ref{construction/universal/}).
    
\end{itemize}

\subsubsection*{Sketches and modules}

The main object of study of this paper is a \emph{sketch with constructions}. 
The name is, of course, derived from the `esquisses' of Ehresmann \cite{ehresmann1968esquisses}.
Sketches with constructions are a formalisation of the idea of $\infty$-categories with \texttt{<relation>}s and \texttt{<construction>}s from the introductory paragraph.
They provide a convenient language for describing the construction of a large class of $\infty$-categories with universal properties, and for presenting objects in terms of a dense subcategory.


\begin{definition}
\label{main/definition}

  A \emph{sketch with constructions} is a triple $(C\sep R\sep K)$ consisting of:
  \begin{itemize}
    \item an $\infty$-category $C$ (whose objects are called \emph{cells} (thinking topologically) or \emph{symbols} (thinking logically));
    \item a set $R$ of \emph{conical diagrams} in $C$ (the \emph{relations});
    \item a set $K$ of \emph{diagrams} in $C$ (the \emph{constructions}). In this paper we assume that the singleton diagrams in $C$ belong to $K$.
  \end{itemize}
  If $K$ is the set of all diagrams in $C$ (within a particular universe which also contains $C$), then we may suppress it from the notation and refer to the pair $(C\sep R)$ simply as a \emph{sketch}.
  We often refer to cells and relations together with a single script symbol $\sk C\defeq (C\sep R)$.
  Categories $\Fun((C\sep R\sep K)\sep(C'\sep R'\sep K'))$ of morphisms between sketches with constructions are defined in the obvious way.
  
\end{definition}

In light of this definition, we can reformulate the universal property from the opening paragraph more precisely as follows:
\begin{quote}

  \emph{$C'$ is the universal $\infty$-category under $C$ that admits $K$-colimits and in which $R$-cones become colimit cones.}
  
\end{quote}
The theory of left Bousfield localisations and local objects of $\infty$-categories \cite{HTT} tell us to construct $C'$ as a full subcategory of a category of $R$-\emph{continuous} presheaves of $\infty$-groupoids on $C$.
These presheaves are the \emph{models} of the sketch.
The structural functor $C\rightarrow C'$ is dense, and left Kan extension provides unique $(R\cup K)$-colimit preserving extensions of functors, and in fact, these extensions preserve all colimits of $R$-continuous presheaves that happen to have apex isomorphic to a $K$-colimit.
The localisation from presheaves to $R$-continuous presheaves can be expressed as a transfinite, convergent iteration of a `plus' construction which generalises the well-known Quillen plus construction and the plus construction for a Grothendieck topology. 

These methods are the bread and butter of the presentable $\infty$-category theorist, and in fact much of the reasoning applies to full subcategories of any Bousfield localisation of a presheaf category.
Nonetheless, presentations in which the weak equivalences and `constructible' objects are defined in terms of diagrams seem to be ubiquitous and important enough to warrant spelling out this line of reasoning in the \emph{sketchable} case, and studying its application to major examples, just as the classical category theorists did for (co)limit sketches on 1-categories.

Here is a selection of the results we get about these categories of modules.\footnote{In this introduction and throughout the paper, we make systematic use of typical ambiguity and avoid specifying universes. This introduction is written in such a way that the reader may choose a favourite universe $\Universe$ and resolve every ambiguous term by inserting $\Universe$. Sets are not assumed $\Universe$-small unless specified.}

\begin{theorem}[Properties of categories of modules]

  Let $(\sk C\sep K_\sk C)$, $(\sk D\sep K_\sk D)$ be small sketches with constructions, $f:(\sk C\sep K_\sk C) \rightarrow (\sk D\sep K_\sk D)$ a functor.
  \begin{enumerate}
    \item \eqref{constructible/universal}.
      If $(\sk D\sep K_\sk D)$ is constructibly cocomplete (Def.~\ref{construction/cocomplete/definition}), restriction along $\Yoneda_\sk C$ induces an equivalence of categories
      \[
        \Fun(\Mod^{K_\sk C}(\sk C) \sep (\sk D\sep K_\sk D)) \quad \tilde\rightarrow \quad 
        \Fun((\sk C\sep K_\sk C) \sep (\sk D\sep K_\sk D)).
      \]
      The inverse equivalence takes a functor $\sk C\rightarrow \sk D$ to its left Kan extension along $\Yoneda_\sk C$.
      This extension preserves all colimits of $\sk C$-modules which are representable in $\Mod^{K_\sk C}(\sk C)$.
      \hfill \emph{Extension}
      
    \item 
      $f$ exhibits $\sk D$ as a category of constructible modules for $(\sk C\sep K_{\sk C})$ if and only if the left Kan extension 
      \[
        f^\dagger \defeq f_!\Yoneda_{\sk C}:\sk D\rightarrow \Mod(\sk C)
      \]
      is fully faithful with image spanned by the $K$-constructible presheaves. \hfill \emph{Recognition}
      
    \item \eqref{essential/criteria}, \eqref{essential/fully-faithful}.
      The restriction functor $f^*:\Mod(\sk D)\rightarrow\Mod(\sk C)$ admits a right adjoint if and only if restriction $f^\dagger:\sk D\rightarrow \Mod(\sk C)$ takes $R(\sk D)$-cones to colimits.
      In this case, $f_!$ is fully faithful if and only if $f^\dagger f$ equals the Yoneda functor $\Yoneda_\sk C:\sk C\rightarrow \Mod(\sk C)$.
      \hfill \emph{Essential adjunction}
      
    \item \eqref{constructible/transitive}.
      Denote by $K-\colimit$ the set of conical diagrams in $\Mod^K(\sk C)$ which are colimits of elements of $K$, and consider $\Mod^K(\sk C)$ as a sketch with relations $ R(\sk C)\cup K-\colimit$.
      Then there is a unique equivalence 
      \[
        \Mod\left( \Mod^K(\sk C) \right) \cong \Mod(\sk C)
      \]
      compatible with the inclusion of $\Mod^K(\sk C)$. \hfill \emph{Transitivity}
      
  \end{enumerate}

\end{theorem}
Versions of some of these results appear in \cite[\S5.3.6]{HTT}, though $K$ is treated differently (cf.~\S\ref{previous-work/}).

\subsubsection*{Homotopy theory of diagrams}

Models of a sketch can be expressed in various ways as colimits of functors $U:I\rightarrow C$, and there are various general tactics for manipulating such diagrams to draw conclusions about $\Mod(\sk C)$.
A systematic way of expressing this is the observation that the $\infty$-category $\Cat\slice C$ of diagrams in $C$ is a \emph{model through localisation} for $\Mod(\sk C)$.
Our expectations for the behaviour of this model are informed by the philosophy of \emph{categorical homotopy theory} exemplified by Quillen's `Theorem A' and codified in Grothendieck's famous `pursuing stacks' manuscript.
Sure enough, there is a localisation $\Cat\slice C\rightarrow\P(C)[W^{-1}]$ whose set of weak equivalences shares the same basic properties as Grothendieck's \emph{localisateurs fondamentals} \cite[\#2.1.1]{maltsiniotis2005theorie}.

\begin{theorem}[Localisers from sketches \eqref{sketch/localizer/properties}] \label{main/localizer}

  Let $\sk C$ be a sketch.
  The fully faithful inclusion $\Mod(\sk C) \hookrightarrow \P(C) \hookrightarrow \Cat\slice C$ admits a left adjoint \emph{classifying module} construction
  \[
    |-|_\sk C:\Cat\slice C\rightarrow \Mod(\sk C)
  \]
  which takes a diagram to its colimit in $\Mod(\sk C)$.
  The set $W_\sk C\subseteq (\Cat\slice C)^{\Delta^1}$ of weak equivalences for this localisation enjoys the following properties:
  \begin{enumerate}
  
    \item
      $W_\sk C$ is strongly saturated.
      
    \item
      If $I$ has a final object $\omega$, then 
      \[
        [\omega:\point\rightarrow I\stackrel{U}{\rightarrow} C] \in W_\sk C
      \]
      for all $U:\Fun(I,C)$. \hfill{Contracts cells}
      
    \item 
      If $I\stackrel{h}{\rightarrow} J\rightarrow K\rightarrow C$ is such that
      \[
        [h\slice k:I\slice k\rightarrow J\slice k\rightarrow C]\in W_\sk C
      \]
      for all $k:K$, then $h\in W_\sk C$. \hfill{Quillen A}
  \end{enumerate}
  Moreover, $W_\sk C$ is \emph{of cellular generation}: it is minimal among strongly saturated sets of maps of diagrams containing cofinal maps and (a small subset of) $W_\sk C\cap(\Cat\coneslice C)$, where $\Cat\coneslice C$ is the set of conical diagrams in $C$, and satisfying Quillen A.
  
  Conversely, any subset of $(\Cat\slice C)^{\Delta^1}$ which is of cellular generation in this sense is the set of weak equivalences for a sketch on $C$.
  
\end{theorem}

Unlike the classical theory, we also have a functoriality statement:

\begin{theorem}[Change of sketch \eqref{essential/localizer}]

  If $f:\sk C\rightarrow \sk D$ is a functor of sketches, then the extension of and postcomposition with $f$ fit into a commuting square
  \[
    \begin{tikzcd}
      \Cat\slice C \ar[r, "f\circ-"] \ar[d] & \Cat\slice D \ar[d] \\
      \Mod(\sk C) \ar[r, "f_!"] & \Mod(\sk D)
    \end{tikzcd}
  \]
  and in particular, $f\circ W_\sk C\subseteq W_\sk D$.
  
  If $f$ is \emph{essential}, then the restriction of cells and comma object functors form a commuting square
  \[
    \begin{tikzcd}
      \Cat\slice D \ar[r, "C\downarrow_D(-)"] \ar[d] & \Cat\slice C \ar[d] \\
      \Mod(\sk D) \ar[r, "f^*"] & \Mod(\sk C)
    \end{tikzcd}
  \]
  and in particular, $C\downarrow_DW_\sk D \subseteq W_\sk C$.

\end{theorem}

\subsubsection*{Classification questions}

Some of the basic questions in the theory of sketches are:
\begin{itemize}

  \item Which localisations of presheaf or diagram categories are sketchable? \hfill(\emph{Sketchability})
  \item Which conical diagrams in $C$ become colimits in $\Mod(C\sep R)$? \hfill($R$-\emph{acyclicity})
  \item Which diagrams in $C$ admit colimits in $\Mod^K(C\sep R)$? \hfill($K$-\emph{constructibility})
  
\end{itemize}
And, unifying and extending the last two questions:

\begin{itemize}

  \item When do two sketches present the same category? \hfill(\emph{constructible Morita equivalence})
  
\end{itemize}

\paragraph{Sketchability}
The first of these questions has a gratifyingly straightforward answer:

\begin{theorem}[Criterion for sketchability]

  A Bousfield localisation of $\P(C)$ is sketchable if and only if it is generated as a Bousfield localisation by morphisms with representable target.
  
\end{theorem}

One consequence of the restriction that weak equivalences be generated by maps to representables is that the localisation functor can be described `concretely' by a kind of plus construction; see \eqref{sketch/module/localizer/explicit}, \cite{anel2020small}.
Another satisfying consequence is the following:

\begin{corollary}[\ref{sketch/on-point/classification}]

  Let $L:\Space\rightarrow\Space$ be a Bousfield localisation.
  The following are equivalent:
  \begin{itemize}
    \item $L$ is presented by a sketch on the one-object category.
    \item $L$ is locally Cartesian.
    \item $L$ is a nullification functor.
    \item $\mathrm{Ho}(L)$ is presented by an aspherically generated Grothendieck fundamental localiser on $\Cat_1$.
  \end{itemize}
  
\end{corollary}

Continuing with the theme that our version of colimit sketches are a kind of indexed $\infty$-categorical analogue of (aspherically generated) fundamental localisers, the substantial work of Cisinski on this subject \cite{cisinski2006prefaisceaux} leads us to a conjecture:

\begin{conjecture}[Sketches from localisers] \label{main/conjecture}

  Let $C$ be an $\infty$-category, $W\subseteq (\Cat\slice C)^{\Delta^1}$ a set of maps satisfying the following properties:
  \begin{enumerate}
  
    \item 
      $W$ is stable under 2-out-of-3 and retracts. \hfill{Weak saturation}
      
    \item
      If $I$ has a final object $\omega$, then 
      \[
        [\omega:\point\rightarrow I\stackrel{U}{\rightarrow} C] \in W
      \]
      for all $U:\Fun(I,C)$. \hfill{Contracts cells}
      
    \item 
      If $I\stackrel{h}{\rightarrow} J\rightarrow K\rightarrow C$ is such that
      \[
        [h\slice k:I\slice k\rightarrow J\slice k\rightarrow C]\in W
      \]
      for all $k:K$, then $h\in W$. \hfill{Quillen A}
  \end{enumerate}
  Suppose further that $W$ is \emph{of cellular generation}: it is minimal among sets of maps of diagrams containing (a small subset of) $W\cap(\Cat\coneslice C)$ (where $\Cat\coneslice C$ is the set of conical diagrams in $C$) and satisfying these three properties.
  Then $W$ is the set of weak equivalences for a sketch on $C$.
  
\end{conjecture}
This statement is much stronger than the converse part of Theorem \ref{main/localizer} because we hypothesise only that $W$ is \emph{weakly} saturated and contains the inclusions of final objects of diagrams (rather than all cofinal functors).

\paragraph{Acyclicity} Our main contribution to the study of \emph{acyclic cones} for a sketch is a set of criteria for a set of relations to be \emph{saturated}, that is, maximal among sets of relations presenting a given category of modules (Proposition \ref{saturated/criteria}).
We relate this property to the saturation properties satisfied by the corresponding sets of weak equivalences on $\P(C)$ (a Bousfield class) and $\Cat\slice C$ (an indexed fundamental localiser) and comment on its application to \emph{aspherical} diagrams and nullification of spaces.

\paragraph{Constructibility}
The main focus of our discussion of constructibility is criteria for a diagram of constructible modules to have a colimit.
We discuss applications of the method of constructing $\sk C$-\emph{rectifications} of diagrams in $\Mod^K(\sk C)$.

A typical application of this method is where we have a way to specify $K$ naturally for a family of categories that includes both $\sk C$ and $\Mod^K(\sk C)$ --- for example, if $K$ is the set of all diagrams whose index category belongs to a fixed set of categories $K(\point)\subseteq\Cat$.
The question of whether 
\[
  (\sk C\sep K_{\sk C})\rightarrow  
  \left( \Mod(\sk C\sep K_\sk C)\sep K_{\Mod(\sk C\sep K_\sk C)} \right)
\]
is a constructible Morita equivalence is then equivalent to the question of whether colimits of shape $K$ are \emph{stable for composition}.
In some cases, one can apply this as a general criterion (Proposition \ref{construction/universal/criterion}) to conclude that $K$-cocompletion is \emph{idempotent}.
We discuss several important cases --- finite discrete, $\kappa$-small, idempotents, filtered, sifted, and $n$-connected diagrams, and descent for higher Deligne-Mumford stacks --- in \S\ref{construction/universal/}.

\subsection{Previous work} \label{previous-work/}

\paragraph{Sketches in classical category theory}

The idea of a general theory of `adjoining colimits with certain colimits enforced' (and perhaps also some limits) has some pedigree in classical category theory.
Rather than attempt to summarise a history with which I am not familiar, I direct the interested reader to the textbook or survey accounts \cite{makkai1989accessible, adamek1994locally, barr1999notes}.
The main thrusts of the classical investigation seem to have been concerned with accessibility of categories of models, adjoint functor theorems, and applications to computer science.

Quite apart from the homotopical context in which the present paper is couched, my work has some points of departure from the thread of the classical literature:
\begin{itemize}
  
  \item 
    I chose to focus on presheaves and colimits rather than functors and limits. This is a purely cosmetic difference.
  
  \item
    The classical literature discusses a question of \emph{sketchability of categories}, the typical answer being Lair's theorem that sketchable $\equiv$ accessible.
    In this paper, we address instead the more focussed question of \emph{sketchability of localisations}; alternatively, sketchability of categories with respect to a chosen dense functor.
    On the other hand, we hardly discuss accessibility.
    
  \item
    Access to $\infty$-categorical methods enables us to make statements about the category of sketches without developing an auxiliary theory (say of $(2,1)$-categories).
    
  \item
    Our notion of \emph{constructibility} (\S\ref{construction/}) seems not to have an analogue in the classical literature.
    
\end{itemize}

\paragraph{Lawvere theories}
One of the major classes of examples --- the finite coproduct sketches --- dates back even further than general \emph{esquisses} to the thesis of Lawvere \cite{lawvere1963functorial}.
Lawvere's ideas have pervaded homotopy theory, finding application in defining homotopy algebraic structures \cite{badzioch2002algebraic, cranch2009algebraic}, motivic homotopy theory \cite{voevodsky2010simplicial}, and derived geometry \cite{DAGV}.

\paragraph{``Pursuing stacks''}
Our notion of sketches can be seen as a massive generalisation of an $\infty$-categorical analogue of Grothendieck's notion of `fundamental localiser' \cite{maltsiniotis2005theorie}.
Passing from $\Cat_1$ to $\Cat_\infty$ sidesteps many of the difficulties involved in Grothendieck's programme --- for example, groupoid completion $\Cat_\infty\rightarrow\Space$ preserves colimits, obviating the need for a theory of homotopy colimits.
However, the computation of limits and mapping spaces is still a difficult problem, and there is thus still a need in higher category theory for a language of `diagrammatic' homotopical algebra.
This paper can be viewed as a contribution to such a language.

\paragraph{Adjoining colimits to $\infty$-categories}

To the best of my knowledge, this is the first work to explicitly formulate the notion of sketch as a first-class object in higher category theory.\footnote{Despite the suggestive title of \cite{tuyeras2017sketches}, which I came across while researching whether or not it is safe to make this claim, the latter is actually a study of presentations of higher category theory using 1-categorical sketches.}
However, there is an obvious comparison to make with the results of \cite[\S5.3.6]{HTT}, which also formulates a general result for adjoining colimits of a certain shape, preserving certain relations.

The main technical difference between this work and that of Lurie is in our treatment of constructibility: where we ask for closure under colimits of certain diagrams, Lurie asks for closure under \emph{iterated} colimits over all diagrams having certain index categories. 
In particular, Lurie's $K$-cocompleteness is defined independently of the starting category of `cells'.
The tradeoff is that while we have more control over presentations of objects in our constructible cocompletions, we lose built-in guarantees on existence of colimits in the category of constructible models; instead, this must be treated as a separate question.

Every Lurie-style cocompletion can be presented by a sketch, but not vice versa.
An important example of a set of constructions that cannot be expressed only in terms of their index categories are those that present the theory of geometric stacks \S\ref{construction/universal/stacks}.



\section{Diagrams} \label{diagrams/}

Sketches are about two things: density, and the homotopy theory of diagrams.
In this section we cover background material on these two things.

We use the $\infty$-category theory of \cite{HTT} and \cite{cisinski2019higher}, with the following caveats:
\begin{itemize}

  \item 
    $\infty$-categories are just called \emph{categories}, except where we want to make a specific emphasis. Classical, set-based categories are called \emph{1-truncated} categories. ``There is a unique $XYZ$'' means ``the space (or simplicial set) of $XYZ$ is contractible.''
    
  \item
    The largest $\infty$-groupoid contained in an $\infty$-category $C$ is denoted $\Object(C)$.
    (It is modelled by the largest Kan complex contained in a quasicategory model for $C$.)
    
  \item
    We define a \emph{set of objects in a category $C$} to be a subset of $\pi_0(\Object(C))$.
    If $S$ is a set or space and $S\rightarrow C$ a map, we abuse language and say also that $S$ is a `set of objects of $C$', conflating it with its image in $\pi_0(\Object(C))$.

\end{itemize}

\subsection{Universe}

The theme of this paper being cocompletions and adjoint functor theorems, we must be explicit about our approach to cardinality issues.
Our approach is based on the axiom of universes.
We do not fix a universe throughout, but instead introduce some syntactical innovations to avoid having to mention it all the time (\emph{typical ambiguity}, \ref{universe/ambiguity}).

\begin{para}[Universe]
\label{size}

  We assume that every set is contained in some Grothendieck universe (or what is the same thing \pathcite{universe/axiom/inaccessible-cardinal}, that every set has rank bounded above by an inaccessible cardinal). For simplicity, we exclude the smallest Grothendieck universe $\omega$ whose elements are finite sets.
  
  Let $\Universe$ be an (uncountable) Grothendieck universe. We make the following definitions:
  \begin{itemize}
    
    \item 
      A $\Universe$-\emph{set} or $\Universe$-\emph{small} set is an element of $\Universe$;
    
    \item 
      A $\Universe$-\emph{space} is a space which has the homotopy type of a CW complex with a $\Universe$-set of cells; equivalently, all of whose homotopy groups are $\Universe$-sets. (The equivalence of these criteria fails for the excluded case $\Universe=\omega$).
    
    \item 
      A $\Universe$-\emph{category} is a category whose space of objects and spaces of morphisms are all $\Universe$-spaces.
      
    \item
      A locally $\Universe$-small category is a category, each of whose mapping spaces is $\Universe$-small. 
    
    \item
      A $\Universe$-\emph{presheaf} on a category $C$ is as defined in \pathcite{cis/presheaf/small/definition}, i.e.~as a presheaf which is equivalent to a $\Universe$-colimit of representables.
      If $C$ is $\Universe$-small, this is the same as being valued in $\Universe$-spaces.
      
      The Yoneda embedding $\Yoneda_C:C\rightarrow\Universe\P(C)$ and covariance of $\Universe\P$ by left Kan extensions of functors are always defined, even if $C$ is not $\Universe$-small.
      However, in this case $\P(C)$ may fail to be presentable, and pullback is not defined.
      
  \end{itemize}
  In the case of $\infty$-categories, these homotopy-invariant conditions are called `essentially $\Universe$-small' in \pathcite{htt/universe}.
  In this paper, I will also write $C \ll \Universe$ (compare \cite[Def.~5.4.2.8]{HTT}).
  By the axiom of universes, every set, space, category, or presheaf belongs to some universe.
  
  In order for sets, spaces, or categories to themselves form a category, it is necessary to choose a universe; hence, for each $\Universe$ we have the categories $\Universe\Set$, $\Universe\Space$, $\Universe\Cat$ (see \S\ref{model/}).

\end{para}

\begin{remark}[The case $\Universe=\omega$]

  The various definitions one can make diverge in the minimal case $\Universe=\omega$, and there isn't really a satisfactory notion of `finite space' that is stable under basic constructions (i.e.~pushout and fibre product) \cite{anel2021elementary}.
  
\end{remark}

\begin{para}[Cocomplete]

  A category is said to be \emph{$\Universe$-cocomplete} if it admits all colimits indexed by $\Universe$-categories.
  It is said to be \emph{cocomplete} if it is $\Universe$-cocomplete for some $\Universe$.
  
  Let $E$ be a cocomplete category. Say that $C$ is \emph{$E$-small} if it is $\Universe$-small for some $\Universe$ which respect to which $E$ is cocomplete; write also $C\ll E$.
  That is, we let any cocomplete category act as a surrogate for the universe with respect to which it is cocomplete. 
  In particular, for any universe $\Universe$, $\Universe$-small is the same as $\Universe\Set$-small.
  
\end{para}

\begin{para}[Presentability]
\label{presentable} 

  An $\infty$-category is $\Universe$-\emph{presentable} if it is equivalent to a Bousfield localization at a $\Universe$-set of morphisms of the category of $\Universe$-presheaves on a $\Universe$-category \pathcite{cis/presentable/definition}. 
  By \pathcite{htt/bousfield-localization/existence}, such a localisation is always $\Universe$-\emph{accessible}; hence the $\infty$-categories we are calling `presentable' also satisfy Simpson's criteria \pathcite{htt/presentable/criterion}.
  In particular, a $\Universe$-presentable category is locally $\Universe$-small (but usually not $\Universe$-small).
  
\end{para}

\begin{para}[Typical ambiguity]
\label{universe/ambiguity}
  
  The universe $\Universe$ may be suppressed from the notation and inferred from the context according to precise rules which I will explain as we go along. This is called \emph{typical ambiguity} after Russell's type hierarchy.
  
  A term which must be prefixed by a specific universe to have meaning is said to be \emph{typically ambiguous}.
  More generally, a term may require several universes.
  Each space where a universe is omitted is called an \emph{implicit universe}.
  Implicit universes can be \emph{resolved} by prepending choices of universe to all such terms.
  Statements involving multiple implicit universes will usually have restrictions (in the form of inequalities or equalities) on the universes that may be used to resolve the ambiguity.
  For all such statements in this paper, I will explicitly declare the rules for resolving ambiguity.
  Each rule is an inequality (strict or non-strict) or an equality between ambiguities in an expression.
  The ambiguity can then be resolved globally by the following algorithm:
  \begin{enumerate}
    \item
      Let $\sigma$ be the set of implicit universes in the document.
      
    \item
      The rules for resolving ambiguity induce the structure of a directed graph on $\sigma$ whose edges are marked with either $<$ or $\leq$.
      A \emph{resolution} of the ambiguity is a map $r$ from the set of vertices of this graph to the set $\{0,1,\ldots,n\}$ for some $n:\N$ such that $r(x)\leq r(y)$ whenever there is a $\leq$-edge connecting $x$ to $y$ and $r(x)<r(y)$ whenever there is a $<$-edge connecting $x$ to $y$.
      This is possible provided there is no directed cycle in the graph containing a $<$-edge.
    
    \item
      Use the resolution $r$ to substitute each implicit universe with an explicit one from a list $\Universe_0\ll\cdots\ll\Universe_n$.
  
  \end{enumerate}
  Systematically using typically ambiguous statements has the benefit of economy of notation and may improve readability, but it carries the risk of the careless user combining ambiguous statements in a way that yields a statement whose ambiguity cannot be resolved.
  For example, he might incorrectly conclude from the fact that $\Sketch\rightarrow\PrL\Cat$ is a colimit-preserving localisation (true, but not proved in this paper) that it admits a right adjoint (false).
  
\end{para}

\begin{para}[Basic rules for resolving ambiguity]  \label{universe/rules}

  The following general principles apply to my use of ambiguous expressions in this document:
  \begin{itemize}
  
    \item 
      The symbols $\Set$, $\Space$, $\Cat$ are all typically ambiguous, as are the words `small', `locally small', `cocomplete', and `presentable', and the presheaves operator $\P$.
      
      The ambiguous symbol $\P(C)$ should be computed with respect to a universe $\Universe$ containing $C$, so that it is $\Universe$-presentable and universal with respect to functors into $\Universe$-cocomplete categories.
      
    \item
      A general principle for resolving implicit universe is that a formula
      \[
        f: A \rightarrow B
      \]
      where both $A$ and $B$ are typically ambiguous (each depending, for simplicity of exposition, on a single universe) should be disambiguated as
      \[
        f:\Universe_A A \rightarrow \Universe_B B
      \]
      where $\Universe_A\leq \Universe_B$.
      If this is the case, we say that the ambiguity can be resolved \emph{from left to right}.
      
      For example, the inclusions $\Set\subset \Space$ and $\Space\subset\Cat$ are resolvable left to right.
  
    \item
      When there are maps in both directions, such as in an adjunction, then usually the disambiguating universe must be the same on both sides.
      In this case we say the expression is resolved against a \emph{fixed} universe.
      
      In particular, expressions in which multiple categories are asserted or supposed to be \emph{presentable} should always be resolved in this way.
      
    \item
      Statements that assert the equivalence of a list of criteria should be formulated to be resolvable against a fixed universe.
  
  \end{itemize}
  
\end{para}


\subsection{Models} \label{model/}

We wish to mainly work internally to the $\infty$-category $\Cat$ of $\infty$-categories i.e.~within the $\infty$-category modelled by the simplicially enriched category of quasi-categories, and \emph{not} directly in the Joyal model category of simplicial sets.
In particular, outside the current section we avoid explicit reference to any simplicial set models.

To make this possible, we need to internalise a couple of basic concepts --- slices and cones --- that are defined using explicit simplicial sets in the literature.

\begin{para}[Models for $\Cat_\infty$ and its basic constructions] \label{model/cat}
Some remarks are in order to compare the commonly-used models for basic concepts in higher category theory with the internal versions we will use.

\begin{itemize}  
  \item    
    The $\infty$-category of $\infty$-categories is denoted $\Cat$. 
    Lurie defines it to be the homotopy-coherent nerve of the simplicial category of quasicategories.
    Using results of Hinich \cite{hinich2015dwyerkan}, it can be identified with the $\infty$-categorical localisation of the category $s\Set$ of simplicial sets at the Joyal categorical equivalences. 
  
  \item
    The $\infty$-category of $\infty$-groupoids is denoted $\Space$.
    Lurie defines it to be the simplicial nerve of the simplicial category of Kan complexes.
    Using results of Hinich \cite{hinich2015dwyerkan}, it can be identified with the $\infty$-categorical localisation of the Kan-Quillen model structure.

    Specialising \pathcite{cis/universe/homotopy-hypothesis} to the case $X=\point$ identifes this $\infty$-category with the quasicategory of right fibrations considered in \emph{op.~cit}.~\cite[Def.~5.2.3]{HCHA}.
    
  \item
    The quasi-category $\Delta^1$ is defined to be the simplicial set represented by $[1]$.
    If $C:\Cat$, we write $C(x,y)$ for the fibre (in $\Cat$) of $C^{\Delta^1}\rightarrow C^2$ at $(x,y)$.
    It happens that this fibre can be computed in $s\Set$; this model is denoted $\mathrm{Hom}_C(x,y)$ in \cite[28]{HTT}.
    By \cite[Cor.~4.2.1.8]{HTT} it is equivalent to $\mathrm{Hom}^{L/R}_C(x,y)$, whence by the comparison theorem \cite[\S2.2.4]{HTT} it models the homotopy type $\Map_C(x,y)$.
    
    In the special case $C=\Fun(I,J)$, we write $\Nat_{I\rightarrow J}(x,y)$ for this space.
    
  \item
    By \pathcite{cis/homotopy-colimit/is-colimit}, Joyal homotopy limits and colimits of quasicategories compute $\infty$-categorical limits and colimits in $\Cat$ (or its slices).
    In particular, by \cite[Cor.~3.3.1.4]{HTT}, strict pullbacks of (co-)Cartesian inner fibrations induce pullbacks in $\Cat$.
    
  \item
    If $A,B$ are quasi-categories, then $A\times B$ is a product of Joyal fibrant objects and hence a product in $\Cat$.
    If $C$ is another quasi-category, $C^B$ is a quasi-category \pathcite{cis/model/functor-category/is-category} invariant under categorical equivalence in $B$ and $C$ \cite[Thm.~3.6.8,9]{HCHA}.
    Moreover, by definition (of $\Cat$), $\Object(C^B)=\Cat(B,C)$ is a mapping space in $\Cat$.
    It follows that the simplicial set exponential of quasi-categories induces an $\infty$-categorical exponential in $\Cat$.

\end{itemize}

\end{para}

\begin{definition}[Slice]
\label{slice/definition}

  Let $C$ be a category, $x:C$ an object. We define the \emph{slice} category $C\slice x$ by the exactness of the pullback square 
  \[
    \begin{tikzcd}
      C\slice x \ar[d] \ar[r] \ar[dr, phantom, "\lrcorner" near start] & 
      C^{\Delta^1} \ar[d, "\target"] \\
      \point \ar[r, "x"] & C
    \end{tikzcd}
  \]
  in the $\infty$-category of $\infty$-categories. 
    
\end{definition}

\begin{para}[Joyal-Lurie overcategories]
\label{slice/joyal-lurie}

  Both \cite{HTT} and \cite{cisinski2019higher} use an explicit quasi-category model of the slice in place of the internal definition above. 
  The following arguments establish that the Joyal construction does indeed model definition \ref{slice/definition}: 
  \begin{itemize}
  
    \item
      The `alternative slice' model $C^{/x}$ is defined in \cite[\S4.2.1]{HTT} to be the value at $(C,x)$ of a right adjoint to the construction 
      \[
        s\Set \rightarrow s\Set_* \quad D \mapsto (D\times\Delta^1)\sqcup_D\Delta^0
      \]
      the pushout being computed in the category $s\Set$ of simplicial sets.
      This is equivalent to the `standard slice' model $C_{/x}$ \cite[Prop.~4.2.1.5]{HTT}.
      
    \item
      By unwinding definitions, $C^{/x}\cong \{x\}\times_CC^{\Delta^1}$ as simplicial sets.
      Since $\target$ is a co-Cartesian fibration \pathcite{htt/comma/source-map-is-cartesian}, this induces a pullback in $\Cat$ \eqref{model/cat}.
        
  \end{itemize}

\end{para}

\begin{para}[Properties of slices]
\label{slice/properties}

  Let us pause to record a couple of properties of this construction:
  \begin{itemize}
  
    \item
      By \eqref{slice/joyal-lurie} and \pathcite{htt/cone/transport} (with $A=\emptyset$ and $B=\point$), $C\slice x$ is equivalent to a right fibration.
    
    \item
      By functoriality of pullback, the construction $(C,x)\mapsto [C\slice x\rightarrow C]$ is a functor from the category of categories equipped with an object (i.e.~$\point\slice\Cat$) into the full subcategory $\mathbf{RFib}$ of $\Cat^{\Delta^1}$ spanned by right fibrations.
      Note that this does not yet give functoriality for morphisms in $C$.
    
    \item
      If $f:x\rightarrow y$ is a map, then the top arrow in the square 
      \[
        \begin{tikzcd}
          (C\slice y)\slice f \ar[r] \ar[d] &  C\slice x \ar[d] \\
          C\slice y \ar[r] & C
        \end{tikzcd}
      \]
      induced by functoriality of slicing is an equivalence.
      Schematically, the map is 
      \[
        \{\begin{tikzcd}\cdot \ar[r] \ar[rr, bend left] & x \ar[r, "f"'] &  y\end{tikzcd}\} 
        \mapsto \{\cdot\rightarrow x\}
      \]
      and its inverse equivalence is given by postcomposition with $f$.
      
      The resulting map $f\circ-:C\slice x\rightarrow C\slice y$ is the unique functor over $C$ that sends $\id_x$ to $f$.
      
    \item
      Suppose $C$ admits pullbacks along $f$. Then $f\circ-$ has a right adjoint $(-\times_yx):C\slice y\rightarrow C\slice x$ whose existence can be seen from the universal mapping property of the pullback.
      
    \item
      The source projection $\source:C\slice x\rightarrow C$ preserves fibre products \pathcite{cis/fibre-product/in-slice} (and in fact, any weakly contractible limit).
        
  \end{itemize}
  
\end{para}

\begin{proposition}[Facts about right fibrations]
\label{fibration/facts}

  Right fibrations enjoy the following properties:
  \begin{enumerate}
    
    \item \label{fibration/facts/slice-criterion}
      A functor $p:E\rightarrow C$ is equivalent to a right fibration if and only if for each $e:E$, the induced map $E\slice e\rightarrow C\slice pe$ is an equivalence of categories.

    \item \label{fibration/facts/left-division}
      Let $h:E\rightarrow F$ be a functor and $p:F\rightarrow C$ (equivalent to) a right fibration such that $ph$ is (equivalent to) a right fibration. Then $h$ is equivalent to a right fibration.

    \item \label{fibration/facts/fibre-product}
      A fibre product in $\Cat\slice C$ of (functors equivalent to) right fibrations is (equivalent to) a right fibration.
      
    \item \label{fibration/facts/slice-category}
      Let $E\rightarrow C$ be (equivalent to) a right fibration. Then $\RFib(E)\cong \RFib(C)\slice E$ via the composition map, where $\RFib(C)\defeq \mathbf{RFib}\times_\Cat\{C\}$ is full subcategory of $\Cat\slice C$ spanned by right fibrations.
      
  \end{enumerate}

\end{proposition}
\begin{proof}

  \begin{enumerate}
    
    \item
      By \pathcite{cis/fibration/right/slice-criterion} and invariance of the statement under categorical equivalence.
      
    \item
      By \ref{fibration/facts/slice-criterion}.
      
    \item
      Let $E:\Lambda^2_2$ be a cospan in $\Cat\slice C$ whose objects are right fibrations. 
      Up to equivalence, we may assume that one (indeed, both) of the maps is an inner fibration, and therefore a right fibration by \eqref{fibration/facts/left-division}.
      Therefore, its fibre product is a homotopy fibre product for the contravariant model structure, and its apex is a right fibration over $C$.  
      
    \item
      By \ref{fibration/facts/left-division}, morphisms in $\RFib(C)$ are themselves right fibrations.
    \qedhere
  
  \end{enumerate}

\end{proof}

\begin{remark}

  In fact, any diagram in $\Cat\slice C$ whose terms are (equivalent to) right fibrations has limit (equivalent to) a right fibration. 
  For example, this can be deduced from the fact \cite[Thm.~3.1.5.1]{HTT} that the identity functor from the contravariant model structure on $s\Set\slice C$ to the slice of the Joyal model structure is right Quillen.
  
  In the interests of avoiding delving too far into model category theoretic considerations, for the present application Proposition \ref{fibration/facts} will suffice.
  
\end{remark}

\begin{para}[Cones] \label{cone}

  Internalising the notion of slice gives us additional flexibility in the study of \emph{cones}.
  A cone is a natural transformation from a functor into a constant functor, i.e.~an object of 
  \[
    \Cone_C \defeq \Fun(I,C)^{\Delta^1}\times_{\Fun(I,C)}C \cong \Fun(I^\triangleright,C).
  \]
  The latter equivalence, considered as a natural isomorphism of functors of $C$, can be taken as the definition of $I^\triangleright$.
  If $U:I\rightarrow C$, write
  \[
    \Cone(U, x) \defeq \Nat_{I\rightarrow C}(U,\underline{x})
  \]
  for the space of cones over $U$ with vertex $x$ (the underline denoting a constant functor).
  
  It will be useful to have another expression for this space of cones.
  Allowing $U$ to vary, we compute
  \begin{align*}
    \Fun(I,C)\slice\underline{x} &= \Fun(I,C)^{\Delta^1}\times_{\Fun(I,C)}\Fun(I,\{x\}) \\
    &= \Fun(I,C^{\Delta^1}\times_C\{x\}) \\
    &= \Fun(I,C\slice x)
  \end{align*}
  from the definition of slice \eqref{slice/definition}.
  Understanding $\Nat_{I\rightarrow C}(U,\underline{x})$ as the fibre over $U$ of the projection $\Fun(I,C)\slice \underline{x}\rightarrow \Fun(I,C)$, we find 
  \[
    \Cone(U, x) \cong \Fun_C((I\sep U)\sep C\slice x).
  \]
  
\end{para}

\begin{para}[Colimit cones] \label{colimit-cone}
  
  A cone $\xi:\Nat_{I\rightarrow C}(F,\underline{x})$ and object $y:C$ induces a commuting square of spaces
  \[
    \begin{tikzcd}
      C(x,y) \ar[r] \ar[d, equals] & \Cone(F, y) \ar[d, equals] \\
      \Fun_C(C\slice x,C\slice y) \ar[r] & \Fun_C(I, C\slice y).
    \end{tikzcd}
  \]
  It is a \emph{colimit} cone if the horizontal maps are invertible for all $y:C$.
  
  A cone over $U$ induces an element of the slice $U\slice C\defeq   (U\slice C^I)\times_{C^I}C$, hence in particular, a vertex of the Joyal-Lurie overcategory model $C_{U/}$.
  This element is a colimit cone in the sense of \cite{HTT, HCHA}.

\end{para}


\subsection{Presheaves and right fibrations} \label{fibration/}

We review here the basic facts we use about the Grothendieck integral/classifying functor constructions and their functoriality.

\begin{para}[Grothendieck integral]
\label{fibration/integral}

  Let $\Cat\slice -:\Cat^\op\rightarrow\Cat$ be the contravariant functor associated by Lurie's `straightening' equivalence to the Cartesian fibration $\target:\Cat^{\Delta^1}\rightarrow\Cat$.
  The \emph{universal Grothendieck construction} for presheaves is a natural transformation
  \[
    \int_-: \P(-) \rightarrow \Cat\slice -.
  \]
  Evaluating at a category $C:\Cat$, it is fully faithful \pathcite{cis/functions-fibrations/integral/is-fully-faithful} with essential image $\RFib(C)$ spanned by the category of Joyal right fibrations \pathcite{cis/functions-fibrations/integral/essentially-surjective}.
  
  For fixed $F:\P(C)$, $\int_CF$ can be constructed as a pullback along $F$ of the universal right fibration $(*\slice\Space)^\op \rightarrow \Space^\op$ \pathcite{cis/fibration/universal/definition}.
  We say then that this square exhibits $F$ as a \emph{classifying presheaf} for $\int_CF$, or that $F$ \emph{classifies} $\int_CF$.
  By \cite[par.~5.8.1]{cisinski2019higher}, the right fibration  $C\slice x$ classifies the representable presheaf $h_x=C(-,x)$.

\end{para}

\begin{proposition}
\label{slice/fibre-product}

  The slice category construction, considered as a functor $C\slice-:C\rightarrow \Cat$, preserves fibre products.

\end{proposition}
\begin{proof}
  
  By definition, the construction is a composite of three functors:
  \[
    C \stackrel{\Yoneda_C}{\longrightarrow} \P(C) 
    \stackrel{\int_C}{\longrightarrow} \Cat\slice C 
    \stackrel{\forget}{\longrightarrow} \Cat.
  \]
  The Yoneda embedding preserves limits by \pathcite{htt/yoneda/preserves-limits}; Grothendieck integration preserves fibre products by Proposition \ref{fibration/facts} part \ref{fibration/facts/fibre-product}; source projection on a slice preserves fibre products \eqref{slice/properties}. \qedhere

\end{proof}

\begin{para}[Relative slices]
\label{slice/relative}

  When $h:C\rightarrow D$ is a functor and $F:D$, we write 
  \[
    C\slice F \defeq C\times_DD\slice F
  \]
  when $h$ can be inferred from the context. 
  We will apply this in particular for $h=\Yoneda_C$ the Yoneda embedding.
  (Note that since pullbacks of right fibrations in the Joyal structure are homotopy pullbacks \eqref{model/cat}, this may be taken to be a literal pullback of simplicial sets.) 
  
  By the base change compatibility of the Grothendieck integral \eqref{fibration/integral}, if $F$ is a presheaf on $C$,
  \[
    \int_CF = C\times_{\P(C)}\int_{\P(C)}\Yoneda_{\P(C)}(F) \cong 
      C\times_{\P(C)}\P(C)\slice F \eqdef C\slice F
  \]
  (although $\Universe\P(C)$ is not $\Universe$-small, the Yoneda functor into $\Universe\P(\Universe\P(C))$ is still defined so this formula resolvable against a fixed universe).
  Again applying base change along $h:C\rightarrow D$, we find:
  \[
    C\slice (F\circ h) \cong C\times_DD\slice F.
  \]
  for $F:\P(D)$.
  In particular, for $d:D$, by the Yoneda lemma we have
  \[
    C \slice h^\dagger d = C \slice d
  \]
  where $h^\dagger d\defeq \Yoneda_D(d)\circ h:\P(C)$ is the pullback to $C$ of the presheaf represented by $d$.

\end{para}

\begin{corollary}[Presheaves and slicing]
\label{slice/presheaf}

  For any category $C$, presheaf $F$ on $C$, and universe $\Universe$, we have $\Universe\P(C\slice F)\cong \Universe\P(C)\slice F$.

\end{corollary}
\begin{proof}

  The case $C,F\ll\Universe$ follows by the Grothendieck equivalence and Proposition \ref{fibration/facts}-\ref{fibration/facts/slice-category}. 
  The identification matches representable objects on the left with morphisms with representable source on the right; hence the sets of $\Universe$-small objects for any $\Universe$.
  \qedhere
  
\end{proof}

\begin{corollary}[Yoneda embedding is dense]
\label{yoneda/is-dense}

  Let $F$ be a presheaf on $C$.
  Then $F$ is a colimit in $P(C)$ of $C\slice F\rightarrow P(C)\slice F$.
  
\end{corollary}
\begin{proof}

  By Corollary \ref{slice/presheaf} and \pathcite{cis/presheaf/colimit-of-yoneda-is-final}. \qedhere
  
\end{proof}

\begin{para}[Right fibrations and homotopy invariance]

  From this point onwards, the term `right fibration' should be taken to mean `an object of $\mathbf{RFib}$' rather than a strict right fibration, i.e.~an object of $\mathbf{RFib}$ which is also an inner fibration.
  This approach seems to be becoming more accepted in higher category theory \cite{gepner2020lax, macpherson2020bivariant}.
  It is justified in the present case because our uses for right fibrations are, firstly, as avatars for presheaves through the equivalence $\P(C)\cong\RFib(C)$, and secondly, as a class abstracting some of the properties of the slice projections $C\slice x\rightarrow C$, which we have also defined homotopy invariantly.
  Proposition \ref{fibration/facts} verifies that the standard stability properties of right fibrations that one takes for granted are, in fact, satisfied for this homotopy-invariant definition.

\end{para}



\subsection{Kan extension}

For the purposes of this paper, the approach to Kan extensions detailed in \cite{HTT}, \cite{cisinski2019higher} need some additional clarification. 
The approach taken here is essentially an $\infty$-categorical reading of \cite[X.3, Thm.~1]{mac2013categories}, and it has the advantage that it is defined in one go for all types of functors with arbitrary source and target categories (the treatment in \cite{HTT} is defined separately for Kan extension along fully faithful and general functors, and that of \cite{HCHA} only addresses the case of target categories which admit enough colimits).

\begin{para}[Lax maps of diagrams]
\label{diagram/lax-map}

  There are (at least) two reasonable ways to make sense of the notion of a map of diagrams: commuting triangles and lax commuting triangles.
  A \emph{(left) lax commuting triangle} $(I,J,C,h,U,V,\psi)$,\footnote{Also called `oplax'.} written graphically
  \[
    \begin{tikzcd}[row sep = tiny]
      I \ar[dd, "h"'] \ar[dr, "U"] \\
      {} \ar[r, phantom, "\Downarrow" marking, near start] & C \\
      J \ar[ur, "V"']
    \end{tikzcd}
  \]
  is the data of functors $U:I\rightarrow C$, $h:I\rightarrow J$, $V:J\rightarrow C$, and a natural transformation $\psi:U\rightarrow Vh$ of functors $I\rightarrow C$.
  Say also that $V$ is a \emph{(left) lax extension} of $U$ along $h$, or that $(h,\psi)$ is a \emph{lax map} from $(I,U)$ to $(J,V)$.
  Lax maps of diagrams in $C$ can be composed by the rule
  \[
    (h',\psi')\circ (h,\psi) \defeq (h'\circ h, \psi'\circ(h^*\psi)).
  \]  
  Because colimits over $I$ are a functor on $\Fun(I,C)$, a lax map of diagrams induces a map of colimits $\colim U\rightarrow\colim V$ when both exist.
  
  In the case that $\psi$ is invertible, we recover the more familiar notion of a commuting triangle, that is a map $\Delta^2\rightarrow \Cat$.
  If an \emph{a priori} lax extension has invertible $\psi$, call it a \emph{strict} extension.

\end{para}

\begin{para}[Lax squares]

  Similarly, we define a \emph{lax square}, depicted
  \[
    \begin{tikzcd}
      I \ar[r, "U"] \ar[d, "h"] \ar[dr, phantom, "\Downarrow" marking] & 
      C \ar[d, "f"] \\
      J \ar[r, "V"] & D
    \end{tikzcd}
  \]
  to be the data of functors $I\rightarrow C\rightarrow D$ and $I\rightarrow J\rightarrow D$ plus a map $\psi:fU\rightarrow Vh$.
  A diagonally transposed lax square is called an \emph{oplax square} (I have chosen this convention quite arbitrarily).
  The lower left half of a lax square yields a left lax map of diagrams $(h,\psi):(I,fU)\dashrightarrow (J,V)$.

\end{para}

\begin{definition}[Kan extension]
\label{kan-extension/definition}

  A lax triangle
  \[
    \begin{tikzcd}[row sep = tiny]
      I \ar[dd, "h"'] \ar[dr, "U"] \\
      {} \ar[r, phantom, "\Downarrow" marking, near start] & D \\
      J \ar[ur, "V"']
    \end{tikzcd}
  \]
  is said to exhibit $V$ as a \emph{left Kan extension} of $U$ along $h$ if for each $j:J$ the induced cone
  \[
    I\slice j \rightarrow D\slice V_j
  \]
  is a colimit cone.

\end{definition}

\begin{remark}

  If 
  \[
    \begin{tikzcd}[row sep = tiny]
      I\slice j \ar[dd, "h"'] \ar[dr, "U"] \\
      {} \ar[r, phantom, "\Downarrow" marking, near start] & D\slice V_j \\
      J\slice j \ar[ur, "V"']
    \end{tikzcd}
  \]
  is a left Kan extension for all $j:J$, then in particular $V_j=\colim[I\slice j\rightarrow D]$, so $V$ is a left Kan extension of $U$ along $h$.
  On the other hand, by \cite[Thm.~6.4.13]{cisinski2019higher} and \pathcite{cis/fibration/right/is-smooth}, the converse holds.

\end{remark}

\begin{lemma}
\label{kan-extension/restrict-codomain}

  Let 
  \[
    \begin{tikzcd}[row sep = tiny]
      I \ar[dd, "h"'] \ar[dr, "U"] \\
      {} \ar[r, phantom, "\Downarrow" marking, near start] & D \\
      J \ar[ur, "V"']
    \end{tikzcd}
  \]
  be a left lax extension, and let $D'\subseteq D$ be a full subcategory containing the images of $V$ and $U$ and whose inclusion preserves colimits (when they exist).
  Then this triangle is a left Kan extension diagram if and only if the triangle
  \[
    \begin{tikzcd}[row sep = tiny]
      J \ar[dd, "h"'] \ar[dr, "V"] \\
      {} \ar[r, phantom, "\Downarrow" marking, near start] & D' \\
      I \ar[ur, "U"']
    \end{tikzcd}
  \]
  obtained by restricting  the codomain to $D'$ is a left Kan extension. \qedhere
  
\end{lemma}
\begin{proof}

  A cone in $D'$ is a colimit if and only if its image in $D$ is. \qedhere
  
\end{proof}

The following result allows us to reduce our notion of Kan extension to the case considered in \cite[\S6.4]{HCHA}.

\begin{proposition}[Existence of Kan extensions]
\label{kan-extension/exist-unique}

  Let $U:I\rightarrow D$, $h:I\rightarrow J$ be two functors.
  The following are equivalent:
  \begin{enumerate}
  \item
    $U$ admits a left Kan extension along a functor $h:I\rightarrow J$.
    
  \item
    The left Kan extension of the composite $U':I\rightarrow D\rightarrow\Fun(D,\Space)^\op$ along $h$ has image in $D$.
    
  \item
    For each $j:J$ the diagram $h\slice j:I\slice j\rightarrow D$ admits a colimit.
  \end{enumerate}
  In this case, the Kan extension is unique.
  
\end{proposition}
\begin{proof}

  By Lemma \ref{kan-extension/restrict-codomain}, by embedding $D$ in $\Fun(D,\Space)^\op$ we may reduce to the case that $D$ is cocomplete.
  Now, by \cite[Prop.~6.4.9]{cisinski2019higher}, restriction $-\circ h:D^J \rightarrow D^I$ admits a left adjoint whose value on $V$ satisfies the condition of Definition \ref{kan-extension/definition}.
  This handles existence.
  
  For uniqueness, take an extension $V:J\rightarrow D$, $\psi:U\rightarrow Vh$ satisfying the condition. 
  By the adjoint mapping property for $h_!\dashv h^*$, there is a unique map $h_!U\rightarrow V$ factoring $\psi$.
  Then the induced maps $\colim_{I\slice j}(U\slice j) \rightarrow (h_!U)_j$, $V_j$ are both isomorphisms, whence $\psi_j:(h_!U)_j \tilde\rightarrow V_j$ for all $j:J$, whence $\psi$ is invertible. \qedhere
  
\end{proof}

\begin{para}[Composition and base change of Kan extensions] \label{kan-extension/base-change}
  
  In light of Proposition \ref{kan-extension/exist-unique}, we are free to leverage the results of \cite{HCHA} on Kan extensions.
  In particular:
  \begin{itemize}
    \item
      Because left Kan extension of functors into a cocomplete category is left adjoint to the restriction functor, and both restriction functors (by associativity of composition) and taking left adjoints (\cite[Prop.~6.1.8]{HCHA}) are compatible with composition, the composite of two left Kan extensions is a left Kan extension.
      
    \item
      By \cite[Prop.~6.4.3]{HCHA}, left Kan extension along a proper functor commutes with base change along any functor, and left Kan extension along any functor commutes with base change along a smooth functor --- for example, along a right fibration \pathcite{cis/fibration/right/is-smooth}.
  \end{itemize}

\end{para}

\begin{para}[Canonical map from a left Kan extension]
\label{kan-extension/canonical-map}

  A consequence of the proof of Proposition \ref{kan-extension/exist-unique} is that if $V':J\rightarrow D$ is another lax extension of $U$ along $h$, then there is a unique map $V\rightarrow V'$ restricting to a factorisation $U\rightarrow hV\rightarrow hV'$. 
  The value of this canonical map at $j:J$ can be understood as the value on the vertex of the map of cones $U\slice j \rightarrow V'\slice j$ which is uniquely determined because $U\slice j:I\slice j\rightarrow D\slice V_j$ is a colimit cone.

\end{para}

\begin{example}[Yoneda pullback as a Kan extension]
\label{yoneda/pullback/as-kan-extension}

  Yoneda pullback $h^*:\P(D)\rightarrow \P(C)$ along a functor $h:C\rightarrow D$ is a left Kan extension of $\Yoneda_C$ along $\Yoneda_D\circ h$: this follows from the formula $C\slice h^*F = C\slice F$ (\ref{slice/relative}) and Corollary \ref{yoneda/is-dense}.

\end{example}


\subsection{Density} \label{dense/}

The property of \emph{density} is central to the relationship between the category of cells of a sketch and its models.
Yet, it seems not to have received much discussion (under this name at least) in the $\infty$-categorical context.
(The related notion of \emph{strong generator} appears in \cite[Def.~5.3.2]{carchedi2020higher}; a collection of objects strongly generates a category if and only if the inclusion of the full subcategory they span is dense.)

\begin{para}[On Yoneda pullback]

  A functor $f:C\rightarrow D$ induces a lax triangle
  \[
    \begin{tikzcd}
      C \ar[r, "f"] \ar[d, "\Yoneda_C"'] & 
      D \ar[dl, "f^\dagger"] \\
      \P(C) & {} \ar[ul, phantom, "\Rightarrow" near end] 
    \end{tikzcd}
  \]
  defined by $f^\dagger\defeq f^*\circ\Yoneda_D$, which we have seen is a left Kan extension.
  This triangle is compatible with composition in $f$ in the following ways:
  \begin{itemize}
    \item 
      If $e:B\rightarrow C$, then $(fe)^\dagger \cong e^*\circ f^\dagger$ by associativity of composition.
      
    \item
      If $g:D\rightarrow E$ admits a right adjoint $h$, then $f^\dagger\circ h \cong (gf)^\dagger$ by the mapping space definition of adjunction \cite[Def.~6.1.3]{HCHA}.
      
  \end{itemize}
\end{para}

\begin{proposition}[Criteria for density]
\label{dense/criterion}

  Let $j:C\rightarrow D$ be a functor. The following conditions are equivalent:
  \begin{enumerate}
    \item \label{dense/criterion/kan-extension}
      $\id_D$ is a left Kan extension of $j$ along itself.
      
    \item \label{dense/criterion/colimit} 
      For each $X:D$, $X$ is a colimit of $j\slice X:C\slice X\rightarrow D$.
      
    \item \label{dense/criterion/counit}
      For each $X:D$, the canonical counit map $\epsilon_X:j_!j^\dagger X\rightarrow X$ is invertible.
      
    \item \label{dense/criterion/pullback} 
      the Yoneda pullback $j^\dagger:D\rightarrow\P(C)$ is fully faithful.
  \end{enumerate}
  If $D$ is cocomplete and $j^\dagger$ is accessible, then under these conditions $D$ is a localisation of $\P(C)$. 
  
\end{proposition}
\begin{proof}

  Conditions \ref{dense/criterion/kan-extension} and \ref{dense/criterion/colimit} are equivalent by definition of Kan extension \eqref{kan-extension/definition}.

  \begin{labelitems}
    \item[\ref{dense/criterion/colimit}$\Leftrightarrow$\ref{dense/criterion/counit}]
      Since $j^\dagger$ is a left Kan extension of $\Yoneda_C$ along $j$ \eqref{yoneda/pullback/as-kan-extension} and $j_!$ preserves colimits, the counit is computed as $        \epsilon_X:\colim_{X:C\slice Y}j(X) \rightarrow Y$.

    \item[\ref{dense/criterion/counit}$\Leftrightarrow$\ref{dense/criterion/pullback}]
      For $X:D$ the composite
      \[
        \P(C)(j^\dagger X,j^\dagger-) \rightarrow \P(D)(j_!j^\dagger X,j_!j^\dagger-) \rightarrow \P(D)(j_!j^\dagger X,-)
      \]
      is an equivalence by adjunction $j_!\dashv j^\dagger$, hence by the (dual) Yoneda lemma $D(X,-)\rightarrow\P(C)(j^\dagger X,j^\dagger-)$ is invertible if and only if $-\circ\epsilon_j$ is invertible. \qedhere
      
  \end{labelitems}

\end{proof}

\begin{definition}

  A functor satisfying the conditions of Proposition \ref{dense/criterion} is said to be \emph{dense}.
  
\end{definition}

\begin{example}

  The canonical example of a dense functor is the Yoneda embedding, which is dense by Corollary \ref{yoneda/is-dense}.
  
\end{example}

\begin{proposition}[Properties of dense functors] \label{dense/properties}

  Let $f:C\rightarrow D$ and $g:D\rightarrow E$ be a string of functors.
  \begin{enumerate}
  
    \item \label{dense/properties/bousfield}
      If $f$ is dense and $g$ is a Bousfield localisation, then $gf$ is dense.
      
    \item \label{dense/properties/localization}
      If $f$ is a localisation and $g$ is dense, then $gf$ is dense.
      In particular, a localisation is dense.
      
    \item \label{dense/properties/extension}
      If $gf$ is dense and $g$ is a left Kan extension of $gf$ along $f$, then $g$ is dense.
    
    \item \label{dense/properties/fully-faithful}
      If $gf$ is dense and $g$ is fully faithful, then $f$ is dense. 
      Moreover, $g$ is a left Kan extension of $gf$ along $f$, and $g$ is dense.
      
    \item \label{dense/properties/left-adjoint}
      A dense left adjoint is a Bousfield localisation.
      
    \item \label{dense/properties/final}
      If $f$ is dense, then any final object of $D$ is a colimit of $f$.
      
  \end{enumerate}
  
\end{proposition}
\begin{proof}

  \begin{labelitems}
    
    \item[\ref{dense/properties/bousfield}]
      We have $ f^\dagger \circ g^*= (gf)^\dagger$ is a composite of fully faithful functors, hence fully faithful.
      
    \item[\ref{dense/properties/localization}]
      $f$ a localisation means that $f^*:\P(D)\rightarrow\P(C)$ is fully faithful (with image the local presheaves); thus $(gf)^\dagger \cong f^*\circ g^\dagger$ is fully faithful.
      
    \item[\ref{dense/properties/extension}]
      \begin{align*}
        \id_E &= (gf)_!(gf) && \text{density of }gf \\
        &=g_!f_!(gf) && \text{composition of Kan extensions} \\
        & =g_!(g) && g=f_!(gf)
      \end{align*}
      
    \item[\ref{dense/properties/fully-faithful}]
      By restriction of codomain for left Kan extensions: we have 
      \[
        gX = \colim_{S:C\slice gX}gfS = \colim_{S:C\slice X}gfS \quad 
        \Rightarrow\quad  X = \colim_{S:C\slice X}fS
      \]
      which proves that $f$ is dense and that $g$ is a left Kan extension; now $g$ is dense by part \eqref{dense/properties/extension}.
      
    \item[\ref{dense/properties/left-adjoint}]
      Let $f$ be a dense left adjoint functor with right adjoint $h$. 
      Then $f^\dagger = \Yoneda_C\circ h$ is fully faithful by hypothesis, whence so is $h$ by left division of fully faithful functors.
      
    \item[\ref{dense/properties/final}]
      Special case of Proposition \ref{dense/criterion}-\ref{dense/criterion/colimit}.\qedhere
      
  \end{labelitems}

\end{proof}

These properties are mnemonically summarised in the following diagrams:
\[
  \begin{array}{cc}
  \eqref{dense/properties/bousfield} \qquad
  \begin{tikzcd}[column sep = tiny]
    \cdot \ar[rr, "\text{dense}"] 
    \ar[dr, "\Rightarrow\;\text{dense}"' very near start] &&
    \cdot \ar[dl, "\text{Bousfield loc.}" very near start] \\ &\cdot
  \end{tikzcd}
  &
  \eqref{dense/properties/localization} \qquad
  \begin{tikzcd}[column sep = tiny]
    \cdot \ar[rr, "\text{loc.}"] 
    \ar[dr, "\Rightarrow\;\text{dense}"' very near start] &&
    \cdot \ar[dl, "\text{dense}" very near start] \\ &\cdot
  \end{tikzcd}
  \\
  \eqref{dense/properties/extension} \qquad
  \begin{tikzcd}[column sep = tiny]
    \cdot \ar[rr] \ar[dr, "\text{dense}"' very near start] &&
    \cdot \ar[dl, "\text{LKE}" at start, "\Rightarrow\;\text{dense}"] \\ &\cdot
  \end{tikzcd}
  &
  \eqref{dense/properties/fully-faithful} \qquad
  \begin{tikzcd}[column sep = tiny]
    \cdot \ar[rr, "\Rightarrow\;\text{dense}"] \ar[dr, "\text{dense}"' very near start] &&
    \cdot \ar[dl, "\text{fully faithful}" at start, "\Rightarrow\;\text{dense, LKE}"] \\ &\cdot
  \end{tikzcd}
  \end{array}
\]

\begin{proposition}[Recognising left Kan extensions along dense functors] \label{dense/extension}

  Let $j:C\rightarrow D$ be a dense functor. 
  If a left Kan extension along $j$ is a strict extension and preserves colimits, then it is unique with this property.
  
\end{proposition}
\begin{proof}

  Let $f':D\rightarrow E$ be a colimit-preserving (strict) extension of $f:C\rightarrow E$.
  Then for all $d:D$, $f'(d) = \colim_{c:C\slice d}f(c)$, whence $f'$ is a left Kan extension of $f$ along $j$. \qedhere
  
\end{proof}

Density is a much stronger condition than generation under colimits:

\begin{example}[Generating, but not dense]

  Every object of $\Delta^\op$ is a colimit of copies of $\Delta^1$ and $\Delta^0$, but the category as a whole does not embed in $\P(\Delta^{\leq1})$ (which is best understood as a category of directed graphs).
  Even worse, $\Delta^{\leq1}\subseteq\Cat$ generates under colimits, but not every category is a colimit of objects of $\Delta^{\leq 1}$ (we need to take \emph{iterated} colimits).

\end{example}

\begin{example}[A full subcategory of a presheaf category that is not dense in its colimit closure]

  Let $C$ be a two element set. 
  The full subcategory spanned by the final object $1_C$ of $\P(C)$ is not dense in its colimit closure $\langle 1_C\rangle \not\simeq \Space$.
  Indeed, it suffices to observe that $\Map(1_C, 1_C\sqcup 1_C)$ is a four element set.

\end{example}

We return to the following concept in \S\ref{construction/}.

\begin{definition}[Constructibility] \label{diagram/constructible}

  Let $C$ be a category, $K$ a set of diagrams in $C$.
  We refer to the pair $(C\sep K)$ as a \emph{category with constructions}; the elements of $K$ are \emph{constructions} in $C$.
  
  Let $h:C\rightarrow D$ be a dense functor. An object $F$ of $D$ is said to be $K$-\emph{constructible} if it is isomorphic to the colimit in $D$ of an element of $K$ composed with $h$.
  Abbreviated name notwithstanding, $K$-constructibility of course depends on the entire tuple $(C\sep D\sep h\sep K)$.
  By the hypothesis on $K$, all objects in the essential image of $h$ are $K$-constructible.
  We also say that a \emph{diagram} $V:J\rightarrow D$ is $K$-constructible if it admits a $K$-constructible colimit in $D$.
  
  Tautologically, any full subcategory $D'$ of $D$ is a set of $K$-constructible objects for some $K$ --- for example the set of all diagrams in $C$ having colimit in $D'$.

  All our sets of constructions will include all singleton diagrams.
  This ensures that all representable presheaves are constructible, so we get a Yoneda functor $\Yoneda_{C,K}:C\rightarrow \P^K(C)$.
  
\end{definition}


\subsection{Classifying presheaves}
\label{shape/}

In this section we will construct a left adjoint localisation to the fully faithful functor
\[
  \int_C:\P(C) \rightarrow \Cat\slice C
\]
which takes a diagram $U:I\rightarrow C$ to $|(I\sep U)|_C = \colim_IU \in \P(C)$. 
This localisation can be thought of as a parametrised version of the \emph{groupoid completion} or \emph{classifying space} functor, or an extension of the classifying presheaf construction \eqref{fibration/integral} to functors that are not right fibrations.

\begin{proposition} \label{shape/criterion}

  Let $C:\Cat$, $F$ a presheaf on $C$, $e:I\rightarrow C\slice F = \int_CF$ a diagram.
  Then the following conditions on $e$ are equivalent:
  \begin{enumerate}

    \item \label{shape/criterion/universal}
      For every right fibration $p:E\rightarrow C$, precomposition with $e$ induces an equivalence of spaces $\Fun_C(I\sep E) \tilde\rightarrow \Fun_C(C\slice F\sep E)$.

    \item \label{shape/criterion/formula}
      Considered as a functor $I\rightarrow\P(C)\slice F$, $e$ is a colimit cone.

  \end{enumerate}

\end{proposition}
\begin{proof}

  By the Yoneda lemma and the identification $\int_C-\cong C\slice -$ \eqref{slice/relative}, the vertical maps in the commuting square
  \[
    \begin{tikzcd}
      \Fun_C(C\slice F \sep C\slice G) \ar[r] \ar[d, equals] &
      \Fun_C(I\sep C\slice G) \ar[d, equals] \\
      \Fun_{\P(C)}(\P(C)\slice F,\P(C)\slice G) \ar[r] &
      \Fun_{\P(C)}(I,\P(C)\slice G)
    \end{tikzcd}
  \]
  are equivalences for all $G:\P(C)$.
  Hence criterion \eqref{shape/criterion/universal} is equivalent to the lower arrow being an equivalence, i.e.~\eqref{colimit-cone} that $e:I\rightarrow \P(C)\slice F$ is a colimit cone.
  \qedhere

\end{proof}

\begin{definition}
  \label{shape/definition}

  A cone $e$ satisfying the conditions of Proposition \ref{shape/criterion} is said to exhibit $F$ as a \emph{classifying presheaf} for $(U\sep I)$.
  We write $|(U\sep I)|_C$, or simply $|U|_C$ if $I$ can be deduced from the context, for a classifying presheaf of $(U\sep I)$.

\end{definition}

\begin{corollary}

  Suppose $U:I\rightarrow C$ is a right fibration. Then the induced map $I\rightarrow \int_C|U|_C$ is an equivalence.
  
\end{corollary}
\begin{proof}

  Immediate from Proposition \ref{shape/criterion}-\ref{shape/criterion/universal}. \qedhere
  
\end{proof}

\begin{proposition}[Canonical enlargement] \label{diagram/canonical}

  The canonical functor $I\rightarrow C\slice|U|_C$ from a diagram to the integral of its classifying presheaf is cofinal.
  
\end{proposition}
\begin{proof}

  Let $p:E\rightarrow C\slice|U|_C$ be a right fibration.
  We must show that restriction along $p$ induces an equivalence on the space of sections of any right fibration over $C\slice |U|_C$ \pathcite{htt/cofinal/definition}.
  Now, $\source\circ p:E\rightarrow C$ is a right fibration, and there is a commuting diagram
  \[
    \begin{tikzcd}
      \Fun_{\int_C|U|_C} \left(\int_C|U|_C, E\right) \ar[d] \ar[r] & 
        \Fun_C \left(\int_C|U|_C, E \right) \ar[d, equals] \ar[r] & 
        \Fun_C \left(\int_C|U|_C, \int_C|U|_C \right) \ar[d, equals] \\
      \Fun_{\int_C|U|_C}(I, E) \ar[r] &
        \Fun_C(I, E) \ar[r] &
        \Fun_C \left(I,\int_C|U|_C \right)
    \end{tikzcd}
  \]
  where:
  \begin{itemize}
      
    \item
      The second and third vertical arrows are invertible by the universal property of $|U|_C$;
    \item
      The left upper, resp.~lower horizontal map are the fibres over $\id_{\int_C|U|_C}$, resp.~$e$, of the right upper, resp.~lower horizontal maps;
  \end{itemize}
  Thus the first vertical arrow is invertible, which is what we had to prove.
  \qedhere
  
\end{proof}

\begin{corollary}[Presheaves as a localisation of diagrams] \label{diagram/presents-presheaves}

  The inclusion $\int_C:\P(C)\tilde\rightarrow \RFib(C)\subseteq(\Cat\slice C)$ admits a left adjoint
  \[
    |-|_C:\Cat\slice C\rightarrow \P(C)
  \]
  that takes a diagram to its classifying presheaf, hence exhibiting $\P(C)$ as a Bousfield localisation of $\Cat\slice C$. 
  It is generated as a localisation by cofinal maps of diagrams, and as a Bousfield localisation by maps of diagrams whose underlying functor of index categories is $1:\Delta^0\rightarrow\Delta^1$.

\end{corollary}
\begin{proof}

  By criterion \eqref{shape/criterion/universal} in the definition of classfiying functor and \pathcite{cis/adjoint/criterion/slice}.
  By Proposition \ref{diagram/canonical}, the unit of the adjunction is cofinal.
  Therefore, by \cite[Lemma 2.2]{macpherson2021locally}, the localisation is generated by cofinal maps.
  
  For the final statement, it suffices to check the local objects. 
  That is, we should show that any diagram that inverts $0:\Delta^0\rightarrow\Delta^1$ is equivalent to a right fibration.
  Let $U:I\rightarrow C$ be such a diagram, $i:I$.
  Then for any $\sigma:c\rightarrow U_i$, the space of fillings
  \[
    \begin{tikzcd}
      \Delta^0 \ar[d, "1"] \ar[r, "i"] & I \ar[d, "U"] \\
      \Delta^1 \ar[r, "{[\sigma:c\rightarrow U_i]}"'] \ar[ur, dotted] & C
    \end{tikzcd}
  \]
  is contractible.
  But this is precisely the fibre over $\sigma$ of the map of right fibrations $I\slice i\rightarrow C\slice U_i$, since this is a fibre of
  \[
    \begin{tikzcd}
      I\slice i \ar[d] \ar[r] \ar[dr, phantom, "\lrcorner" near start] 
      & I^{\Delta^1} \ar[r, "\target"] \ar[d] & I \ar[d, equals] \\
      C\slice U_i \ar[r] & C \downarrow I \ar[r, "\target"] & I \\
    \end{tikzcd}
  \]
  over $i$.
  \qedhere
  
\end{proof}

The following reformulation is based on a comparison with \cite[\#2.1.1]{maltsiniotis2005theorie}.

\begin{lemma}[Classifying presheaves and colimits] \label{shape/quillen-a}

  Let $I\stackrel{h}{\rightarrow} J \stackrel{V}{\rightarrow} C$ be a map of diagrams.
  Then $\colim_{j:J}|hV\slice j| \cong |hV|_C$.
  
\end{lemma}
\begin{proof}

  The colimit of $U$ can be computed first by left Kan extension along $h$. \qedhere
  
\end{proof}

\begin{corollary}[Indexed fundamental localiser] \label{diagram/localizer}

  Let $W_{C}\subseteq \Cat\slice C^{\Delta^1}$ be the set of arrows sent to equivalences by $|-|_C$. Then:
  \begin{enumerate}
    \item 
      $W_{C}$ is strongly saturated.
    
    \item 
      If $I$ has a final object $\omega$, then 
      \[
        [\omega:\point\rightarrow I\stackrel{U}{\rightarrow} C] \in W_C
      \]
      for all $U:\Fun(I,C)$.
      
    \item 
      If $I\stackrel{h}{\rightarrow} J\rightarrow K\rightarrow C$ is such that
      \[
        [h\slice k:I\slice k\rightarrow J\slice k\rightarrow C]\in W_C
      \]
      for all $k:K$, then $h\in W_C$.
  \end{enumerate}
  Moreover, $W_C$ is the minimal strongly saturated class containing cofinal maps of diagrams.
  
\end{corollary}
\begin{proof}

  The first two properties are obvious, the third a consequence of Lemma \ref{shape/quillen-a}.
  Minimality is a reformulation of Corollary \ref{diagram/presents-presheaves}.\qedhere
  
\end{proof}

\begin{para}[Classifying presheaf as a left Kan extension] \label{shape/kan-extension}

  Let $h:C\rightarrow D$ be a functor. Taking left adjoints to the commutative square
  \[
    \begin{tikzcd}
      \P(C) \ar[r, "\int_C"] & \Cat\slice C \\
      \P(D) \ar[u] \ar[r, "\int_D"] & \Cat\slice D \ar[u] 
    \end{tikzcd}
  \]
  yields an identification $h_!|U|_C \cong |hU|_D$ for any diagram $U:I\rightarrow C$. (Here we use that $h\circ-:\Cat\slice C\rightarrow\Cat\slice D$ is left adjoint to pullback along $h$.)

  In particular, taking $(U\sep I)=(\id_C\sep C)$, we obtain
  \[
    |h|_D \cong h_!|\id_C|_C = h_!\underline{\point}_C
  \]
  that is, the shape of the diagram $h$ is the left Kan extension along $h$ of the point presheaf on $C$.
  This formula parallels the notion of \emph{shape} used in higher topos theory.

\end{para}

\begin{remark}[Functoriality with respect to lax maps of diagrams]

  From its definition as a colimit it is clear that one obtains a map of classifying presheaves also from a \emph{lax} map of diagrams \eqref{diagram/lax-map}.
  This can be made fully functorial (i.e.~coherently compatible with composition) as follows.
  Recall from \cite[\S2.8]{macpherson2020bivariant} that there is a \emph{universal evaluation map}
  \[
    \int_{I:\Cat}\left(\P(C)\otimes \P(I^\op)\right) 
    \times_\Cat \int_{I:\Cat} \P(I) \rightarrow \P(C)
  \]
  which composed with the final section $\point_-:\Cat\rightarrow\int_{I:\Cat}\P(I)$ yields
  \[
    \int_{I:\Cat}C^I \rightarrow \P(C).
  \]
  This functor extends the one constructed in \ref{diagram/presents-presheaves} along the inclusion $\Cat\slice C\subseteq \int_{I:\Cat}C^I$.
  We don't need to use this construction in this paper.

\end{remark}

\begin{remark}[Thomason model structures on diagrams]

  By \cite{macpherson2021locally}, the localisation $|-|_C:\Cat\slice C\rightarrow \P(C)$ admits a `Thomason-like' left Bousfield localisation model structure \cite[Ex.~2.12]{mazelgee2015modela} in which all morphisms are cofibrations.
  This is precisely the image in $\Cat\slice C$ of the contravariant model structure of \cite[\S2.1.4]{HTT},
  and the Quillen adjunction between $(s\Set\slice S\sep \mathbf{Joyal})$ and $(s\Set\slice S\sep \mathrm{cov})$ exhibited in \cite[\S3.1.5]{HTT}.
  Any map $h:(I\sep U)\rightarrow (J\sep V)$ obtained as a pullback
  \[
    \begin{tikzcd}
      I \ar[r] \ar[d] & C\slice |U|_C \ar[d] \\
      J \ar[r] & C\slice |V|_C
    \end{tikzcd}
  \]
  is a fibration for this structure.
  Unlike the case $C=\point$, it does not follow that this pullback is preserved by $|-|_C$ --- the localisation is not locally Cartesian.
 
\end{remark}

\begin{remark}[Models for the classifying presheaf]

  For any quasicategory $S$, the diagram
  \[
    \begin{tikzcd}
      \RFib(S) \ar[r, "\id_{s\Set\slice S}"] \ar[d, "L_\mathrm{contravariant}"] & q\Cat\slice S \ar[d, "L_\mathbf{Joyal}"] \\
      \P(S) \ar[r] & \Cat\slice S
    \end{tikzcd}
  \]
  is commutative, where the upper right arrow is the right Quillen functor $(s\Set\slice S\sep \mathrm{cov})\rightarrow (s\Set\slice S\sep \mathbf{Joyal})$ and the vertical arrows are the associated localisation functors \cite{hinich2015dwyerkan}.
  Since Quillen adjunctions of models induce adjunctions of the localisations \cite{mazelgee2015quillen}, taking left adjoints yields a corresponding square
  \[
    \begin{tikzcd}
      \RFib(S) \ar[d, "L_\mathrm{contravariant}"] & q\Cat\slice S \ar[d, "L_\mathbf{Joyal}"] \ar[l, "Q"] \\
      \P(S) & \Cat\slice S \ar[l, "|-|_C"]
    \end{tikzcd}
  \]
  where $Q$ is a fibrant replacement functor for the contravariant model structure.
  
\end{remark}


\subsection{Diagrams and pullback} \label{lax-pullback}

It is not generally true that the naturality square \eqref{shape/kan-extension} is Beck-Chevalley, that is,
\[
  \begin{tikzcd}
    \P(C) \ar[d] \ar[dr, phantom, "!" description] & 
    \P(D) \ar[d] \ar[l] \\
    \Cat\slice C & \Cat\slice D \ar[l, "-\times_DC"]
  \end{tikzcd}
\]
is \emph{not} commutative. 
In this subsection, we will show how replacing pullback with `lax' pullback amends this deficiency.
We take this opportunity to discuss the related notion of \emph{rectification} of diagrams.

\begin{para}[Comma categories]
\label{comma}

  Let $f:C\rightarrow D$ be a functor. We define $D\downarrow C$ by the fibre product
  \[
    \begin{tikzcd}
      D\downarrow C \ar[r] \ar[d] \ar[dr, phantom, "\lrcorner" near start]
        & D^{\Delta^1} \ar[d, "\target"] \\
      C \ar[r] & D.
    \end{tikzcd}
  \]
  It can also be written $D\downarrow f$.
  By \pathcite{htt/comma/source-map-is-cartesian}, the projection $\source:D\downarrow C\rightarrow D$ is a Cartesian fibration.
  (As discussed above \eqref{slice/joyal-lurie}, this pullback may be computed strictly in the category of simplicial sets.)

  On the other hand, taking $C=\{d\}$ recovers the definition of $D\slice d$ from \ref{slice/definition}.
  Similarly, $C\downarrow D$ is defined by replacing target with source projection in the defining pullback square.
  
  More generally, if $U:I\rightarrow D$ is a diagram, then we write
  \[
    f\downarrow U\defeq C\downarrow_DI \defeq (C\downarrow D)\times_DI 
    = C\times_DD^{\Delta^1}\times_DI.
  \]
  The functor $\mathrm{pr}_C:f\downarrow U\rightarrow C$ is called the \emph{lax pullback diagram} of $(I\sep U)$.
  It has a universal property with respect to lax squares
  \[
    \begin{tikzcd}
      \cdot \ar[d] \ar[r] & C \ar[d, "f"] \\
      I \ar[r, "U"'] & D \ar[ul, "\Downarrow" marking, phantom]
    \end{tikzcd}
  \]
  which can be deduced directly from its construction as a fibre product.\footnote{Compare \cite[Def.~3.3.15]{riehl20152}.}

\end{para}

\begin{para}[Lax pullback cone]
\label{lax-pullback}

  Let $U_\omega:C$ be an object, $U:I \rightarrow C\slice U_\omega$ be a cone with apex $U_\omega$, $f:X\rightarrow U_\omega$ a map in $C$.
  Considering $f$ as a functor $f:C\slice X\rightarrow C\slice U_\omega$ and $U$ we obtain a \emph{lax pullback cone}
  \[
    f\downarrow U  \defeq (C\slice X)\downarrow_{C\slice U_\omega}I \stackrel{\source}{\longrightarrow} C\slice X.
  \]
  
\end{para}

\begin{proposition}[Pullback and classifying presheaves] \label{shape/pullback}

  Let $f:C\rightarrow D$ be a functor, $(I\sep U)$ a diagram in $D$.
  \begin{enumerate}
    \item \label{shape/pullback/lax}
      We have $|f\downarrow U|_C \cong f^\dagger|U|_D$.
      
    \item \label{shape/pullback/right-fibration}
      If $f$ is a right fibration, then $|U\times_DC|_C\cong |U|_D\circ f$ in $\P(C)$.
      
    \item \label{shape/pullback/internal}
      If $f:X\rightarrow Y$ is a morphism in $D$ and $U:I\rightarrow D\slice U_\omega$ is a conical diagram, then $|f\downarrow U|_D\cong X\times_Y|U|_D$ in $\P(D)$.
      
  \end{enumerate}

\end{proposition}
\begin{proof}

  \begin{enumerate}
  
    \item
      Target projection $C\downarrow_DI\rightarrow I$ is a co-Cartesian fibration, so the pullback may be computed fibre-by-fibre:
      \[
        |f\downarrow U| \cong \colim_{i:I}|f\downarrow(U_i)| \cong \colim_{i:I}f^\dagger U_i.
      \]

    \item
      Calculate:
      \begin{align*}
        f^*U_!1_I &\cong (U\times_DC)_!f^*1_I && \text{base change along right fibration \eqref{kan-extension/base-change}}\\
        &\cong (U\times_DC)_!1_{I\times_DC} && \text{pullback preserves }1
      \end{align*}
      
    \item
      Apply \ref{shape/pullback/right-fibration} to $D\slice X\rightarrow D\slice Y$ (which is a right fibration by Proposition \ref{fibration/facts}-\ref{fibration/facts/left-division}). \qedhere
  
  \end{enumerate}

\end{proof}

\begin{definition}[Rectification] \label{rectification/definition}

  Let $f: C\rightarrow  D$ be a dense functor and $(J\sep V)$ a diagram in $D$.
  An $f$-\emph{rectification} of $(J\sep V)$ is a lax square 
  \[
    \begin{tikzcd}
      I \ar[d, "h"] \ar[r, "U"] \ar[dr, phantom, "\Downarrow" marking] & C \ar[d, "f"] \\
      J \ar[r, "V"] & D
    \end{tikzcd}
  \]
  such that $(h\sep\psi):(I\sep fU)\dashrightarrow(J\sep V)$ is a left Kan extension.
  The diagram $(U\sep I)$ is the \emph{underlying $C$-diagram} of the rectification.
  This entails $f_!|U|_C\cong |V|_D$.
  
\end{definition}

\begin{remark}[Comma object as universal rectification]

  In the situation of Definition \ref{rectification/definition}, every diagram $(J\sep V)$ in $D$ admits a canonical $f$-rectification given by the comma object $\source:f\downarrow V\rightarrow C$.
  The universal property of comma objects means that it receives a unique map from any rectification, and it satisfies the additional property that it models $f^*|V|_D$.
  It is often useful to be able to construct rectifications `smaller' than the canonical one.

\end{remark}


\section{Sketches}
\label{sketch/}

In this section, we introduce our notion of sketches in $\infty$-category theory.
We discuss both sides of the localisations $\P(C)\rightleftarrows\Mod(\sk C)$ and $\Cat\slice C\rightleftarrows \Mod(\sk C)$ for a sketch $\sk C=(C\sep R)$ and functoriality for change of sketch including Beck-Chevalley conditions.
We end with some discussion of Morita equivalences of sets of diagrams, an application to aspherical diagrams, and some examples of sketches on a point.

\subsection{Conical diagrams}
\label{conical-diagram/}

In this section, we discuss sizes and operations on categories with sets of diagrams.

\begin{para}[Sets of objects of categories]
\label{sets/of-objects}
  
  The set of $\Universe$-sets of objects of a category $C$ --- that is, $\Universe$-small subsets of $\pi_0(\Object(C))$ --- is denoted $\Universe\Power(C)$.
  It is a distributive lattice which admits all meets and $\Universe$-small joins.
  If left ambiguous, assume $C\ll\Universe$.

\end{para}

\begin{para}[Conical diagrams]
\label{conical-diagram/set}

  Let $C:\Cat$ be a category. 
  A \emph{conical} $\Universe$-\emph{diagram} in $C$ is a functor $U:J\rightarrow C\slice c$, where $J$ is a $\Universe$-category, called the \emph{index category} of $(J,U)$.
  Conical diagrams in $C$ form a category 
  \[
    \Universe\Cat \coneslice C \defeq \int_{I:\Universe\Cat,c:C}\Map(I, C\slice c).
  \]
  When left ambiguous, assume $C\ll\Universe$, so that $\Cat\coneslice C$ is presentable.

  A \emph{set of conical diagrams} in $C$ is a set of objects of $\Cat\coneslice C$ (see \ref{sets/of-objects}).
  There are \emph{two} implicit universes that appear when quantifying over sets of conical diagrams: the size of the set, and the size of the (index categories of the) diagrams.
  A $\Universe$-set of $\Universe$-diagrams is said to be $\Universe$-\emph{small}.

\end{para}

\begin{para}[Transferring sets of diagrams] \label{conical-diagram/functoriality}

  A functor $f:C\rightarrow D$ induces a transformation $f_!:\Universe\Cat\coneslice C\rightarrow \Universe\Cat\coneslice D$ of Cartesian fibrations over $\Universe\Cat$, for any $\Universe$.
  If $C$ and $D$ are $\Universe$-small, the fibres of this functor are subcategories of functor categories between small categories, hence small.
  We therefore obtain pullback $f^{-1}$ and left adjoint image $f_!$ operations 
  \[
    f_!:\Power(\Cat\coneslice C) \rightleftarrows \Power(\Cat\coneslice D) : f^{-1}
  \]
  on sets of (conical) diagrams in $C$, $D$.
  
\end{para}

\begin{para}[Boundary morphisms of a conical diagram] \label{conical-diagram/boundary}

  Let $U:I\rightarrow C\slice U_\omega$ be a conical diagram in $C$.
  The \emph{coboundary morphism} of $U$ is the induced morphism
  \[
    d_F:\colim_{i:I}U_i \rightarrow U_\omega
  \]
  in $\P(C)$.
  This provides a mapping
  \[
    \Object(\Cat\coneslice C) \rightarrow \Object(\P(C)^{\Delta^1}).
  \]
  If $R$ is a set of conical diagrams in $C$, write $\bar{R}\subseteq\P(C)^{\Delta^1}$ for the image of $R$ under this mapping.
  (In the dual case of a left conical diagram, use the term \emph{boundary} morphism.)

\end{para}

\begin{para}[Slices] \label{conical-diagram/slice}

  If $R$ is a set of conical diagrams on a category $C$, $c:C$, we write $R(c)\subseteq R$ for the set of diagrams with cone point isomorphic to $c$.
  If $f:C\rightarrow D$ is a functor then we obtain maps $R(c)\rightarrow f_!R(fc)$.
  If $f$ is a right fibration and $R$ is the preimage of a set of diagrams $R'$ on $D$, then by Proposition \ref{fibration/facts}-\ref{fibration/facts/slice-criterion}, this map is a bijection:
  \[
    (f^{-1}R')(c) = R'(fc).
  \]

\end{para}

\subsection{Sketches}  \label{sketch/init/}

We introduce the basic constructions of sketches: trivial and cotrivial, canonical and subcanonical, diagram and localisation sketches, slices, extensions, and exterior (box) products.

\begin{definition}[Sketch]
  
  \label{sketch/definition}

  A \emph{colimit sketch}, or simply \emph{sketch}, is a pair $(C\sep R)$ where $C$ is a category, the category of \emph{cells} of the sketch, and $R$ is a set of conical diagrams $I\rightarrow C\slice c$, called the \emph{relations} of the sketch.
  Say also that $(C\sep R)$ is a sketch \emph{on $C$}.
  A $\Universe$-\emph{sketch} or $\Universe$-\emph{small} sketch is a $\Universe$-category with a $\Universe$-set of $\Universe$-diagrams.
  
  We usually use curly letters $\mathcal{C}$, $\mathcal{D}$, \ldots~to denote sketches.
  The category of cells of a sketch $\sk C$ is denoted $\symb(\sk C)$; for economy of notation, we will often denote it simply by the corresponding Roman character (e.g.~$C\defeq\symb(\sk C)$ in this case). The set of relations on $\sk C$ is denoted $R(\sk C)$.

  A \emph{morphism of sketches} $\sk C\rightarrow \sk D$ is a functor $\symb(\sk C)\rightarrow \symb(\sk D)$ which maps $R(\sk C)$ into $R(\sk D)$.
  We write 
  \[
    \Fun(\sk C, \sk D)\subseteq \Fun(\symb(\sk C),\symb(\sk D))
  \]
  for the full subcategory spanned by the maps of sketches.
  These subcategories are stable for composition of functors.
  
\end{definition}

\begin{definition}[Categories of sketches]
\label{sketch/size}

  A sketch $\sk C$ is said to be $\Universe$-\emph{small}, or a $\Universe$-\emph{sketch}, if $\symb(\sk C)$ is $\Universe$-small and $R(\sk C)$ is $\Universe$-small as a set of conical diagrams \ref{conical-diagram/set}.
  The category of $\Universe$-sketches is denoted
  \[
    \Universe\Sketch \defeq \int_{C:\Universe\Cat} \Universe\Power(\Universe\Cat\coneslice C).
  \]  
  We usually use the typically ambiguous symbol $\Sketch$.
  In this paper, we will mainly focus on properties of \emph{individual} sketches, and to a certain extend the lattice of sketches on a fixed category, and defer to a later work most discussion of categories of sketches.
  
\end{definition}

\begin{example}[Sketches on the empty set]

  We have $\Cat\coneslice\emptyset = \emptyset$, so there is a unique sketch $R=\emptyset$ on the empty category.

\end{example}

\begin{example}[Trivial] \label{sketch/example/trivial}

  Any category $C$ carries a tautological sketch $C_\emptyset$ with no relations.
  It is initial in the lattice $\Power(\Cat\coneslice C)$ of sketches on $C$; hence, $C\mapsto C_\emptyset$ defines a left adjoint to the forgetful functor $\Sketch \rightarrow \Cat$.
  It can be realised as the image under $\emptyset\rightarrow C$ of the unique sketch on $\emptyset$, and it is stable for image and preimage along any functor between categories of cells.

\end{example}

\begin{example}[Sketches on a point]

  Sketches on the one-object category $\point$ are sets of categories.
  See \S\ref{sketch/on-point/}.
  
\end{example}

\begin{example}[Cotrivial] \label{sketch/example/cotrivial}

  Any category $C$ carries a tautological `cotrivial' sketch $(C,\Universe\Cat\slice C)$ depending on a universe $\Universe$.
  It is stable under arbitrary pullback of functors of categories of cells; in particular, the cotrivial sketch on $C$ is a pullback of the cotrivial sketch on $\point$.
  The set of relations is not small, so it does not define a right adjoint to $\forget^R$ (but it would if we allowed large sets of small relations).
  
\end{example}

\begin{para}[Canonical] \label{sketch/canonical/definition}

  If $C$ is $\Universe$-cocomplete, write $\Universe-\mathtt{colimit}\slice C\subseteq \Universe\Cat\slice C$ for the set of all $\Universe$-small colimit cones in $C$.
  We refer to this set of diagrams as \emph{canonical} (though it depends on $\Universe$).
  Note that it is never $\Universe$-small, although it can in good cases be \emph{Morita equivalent} (Def.~\ref{sketch/morita-equivalence/definition}) to a small one.
  
  We will employ typical ambiguity with the term $\Universe-\colimit$ only in certain restricted contexts.
  One such is within the expression $\Fun(\sk C, (D\sep\mathtt{colimit}))$, which can be resolved against any $\Universe$ such that $R(\sk C)\subseteq \Universe\Cat\slice C$.
  Then, it is simply the category of functors that take the elements of $R(\sk C)$ to colimit cones in $D$; in other words, it is actually \emph{unambiguous}.
  
\end{para}  

\begin{para}[Subcanonical] \label{subcanonical/definition}

  A set of conical diagrams $R$ is said to be \emph{subcanonical} if all of its elements are colimit cones.
  Equivalently, it is a subset of $\Universe-\mathtt{colimit}\slice C$ for some $\Universe$.
  (This notion is called `realised' in \cite[40]{makkai1989accessible}.)
  
\end{para}

\begin{para}[Slice sketch]

  Let $\sk C$ be a sketch, $F:\P(C)$.
  The \emph{slice sketch} $\sk C\slice F$ is the sketch whose category of cells is $C\slice F$ and whose relations are the preimage of $R(\sk C)$.
  Since $C\slice F\rightarrow C$ is a right fibration, for any $X\rightarrow F$ the set of relations with vertex $X$ is the same in $C$ and $C\slice F$: $R_\sk C(X) = R_{\sk C/F}(X\rightarrow F)$ \eqref{conical-diagram/slice}.
  
\end{para}

\begin{para}[Extending sketches] \label{sketch/extension}

  Given a sketch $\sk C=(C\sep R)$ and a set $R'\subseteq\Cat\coneslice C$ of conical diagrams in $C$, we introduce a notation for \emph{extending $\sk C$ by $R'$}:
  \[
    \sk C\cup R' \defeq (C\sep R\cup R').
  \]

\end{para}

\begin{para}[Exterior products] \label{sketch/exterior-product}
 
  Let $C$, $D:\Cat$, and let $S\subset \Cat\coneslice C$, $T\subset \Cat\coneslice D$ be sets of diagrams.
  The \emph{exterior product} $S\boxtimes T\subset\Cat\coneslice(C\times D)$ is defined as follows:
  \[
    S\boxtimes T \defeq (S\times D)\sqcup (C\times T)
  \]
  where $S\times D\subseteq\Object((\Cat\slice C)\times D) \rightarrow \Object(\Cat\slice (C\times D))$ is defined using the co-ordinate slices
  \[
    U:I\rightarrow C\slice c\sep d \quad \mapsto \quad (U,\underline{d}):I\rightarrow (C\times D)\slice(c,d)
  \]
  and similarly for $C\times T$.
  Compare \cite[Not.~4.8.1.7]{HA}.
  This defines a binary operation on sketches with unit $(\point,\emptyset)$.

\end{para}

\begin{example}[Diagram]
\label{sketch/example/conical}

  A conical category $I^\triangleright$ carries a natural \emph{conical sketch} $I^\triangleright_{\{\id\}}$ with $R=\{\id_{I^\triangleright}\}$.
  The set of diagrams in a colimit sketch $\sk C$ is exactly the set of morphisms $(I^\triangleright\sep\{\id\})\rightarrow \sk C$ from `diagram' sketches.
  
\end{example}

\begin{example}[Localisation] \label{sketch/example/localization}

  A set of arrows $W\subseteq\Fun(\Delta^1,C)$ can be thought of as a set of conical diagrams via $(\Delta^0)^\triangleright = \Delta^1$.
  Hence the category of \emph{relative $\infty$-categories} \cite{barwick2012relative} embeds into $\Sketch$ as a full subcategory.

\end{example}


\subsection{Modules over sketches}
\label{sketch/module/}

In this section we define the category of \emph{continuous presheaves}, \emph{models}, or \emph{modules} over a sketches, and describe the behaviour under with respect to trivial, cotrivial, (sub)canonical, diagram and localisation sketches, slices, extensions, and box products.

\begin{proposition}[Criteria for $R$-continuity]
\label{sketch/module/criterion}

  Let $\sk C$ be a sketch, $F:C^\op\rightarrow\Space$ a presheaf.
  The following conditions are equivalent:
  \begin{enumerate}
  
    \item \label{module/criterion/internal}
      $F$ defines a morphism of sketches $\sk C\rightarrow (\Space^\op,\mathtt{colimit})$.
      
    \item \label{module/criterion/explicit}
      For every element $U:I\rightarrow C\slice\omega $ of $R$, the boundary map
      \[
        \partial_U:F(U_{\omega}) \rightarrow \lim_{i:I}F(U_i)
      \]
      is an isomorphism.
    
    \item \label{module/criterion/local}
      $F$ is $\bar{R}$-local.

  \end{enumerate}
  Moreover, these conditions are stable under limits in $\P (C)$ and preserved by pullback along morphisms of sketches with small category of cells.
  
\end{proposition}
\begin{proof}
  \begin{labelitems}
  
    \item[\ref{module/criterion/internal}$\Leftrightarrow$\ref{module/criterion/explicit}] 
      Immediate from the definition of morphisms of sketches.
      
    \item[\ref{module/criterion/explicit}$\Leftrightarrow$\ref{module/criterion/local}]
      The extended Yoneda lemma \pathcite{cis/yoneda-lemma} provides an identification 
      \[
        F(U_j) \cong \Nat_{C^\op\rightarrow\Space}(\Yoneda_C (U_j),F)
      \]
      for any $j:I^\triangleright$.
      Using this to translate the continuity condition \ref{module/criterion/explicit}, we find
      \[
        \Nat_{C^\op\rightarrow\Space}(\Yoneda_C(U_\omega)\sep F) \tilde\rightarrow 
          \Nat_{C^\op\rightarrow\Space} \left(\colim_{i:I}\Yoneda_C(U_i)\sep F \right)
      \]
      which is precisely the statement of locality for $F$.
      
    \item[Limit stable]  
      Condition \ref{module/criterion/explicit}, and hence also the other criteria, is stable under limits in $\P(C)$ because these are computed pointwise \pathcite{htt/limit/of-functors}.

    \item[Pullback]
      Let $f:\sk C\rightarrow \sk D$ be a morphism of sketches and let $F:\P(D)$ satisfy \ref{module/criterion/explicit}.
      Then for any $(I\sep U):R(\sk C)$, $F$ takes $fU:I\rightarrow D$ to a limit cone, whence $f^{-1}F$ also satisfies criterion \ref{module/criterion/explicit}.
      \qedhere

  \end{labelitems}
  
\end{proof}

\begin{definition}[Modules]
\label{sketch/module/definition}

  A presheaf on $C$ is said to be $R$-\emph{continuous} if the equivalent conditions of Proposition \ref{sketch/module/criterion} are satisfied.
  Say also that $F$ is a \emph{(left) module over} or \emph{model of} the sketch $\sk C$.
  The full subcategory of $\Universe\P(C)$ spanned by the $R$-continuous presheaves is denoted $\Universe\P_R(C)$, or writing $\sk C\defeq (C\sep R)$, $\Universe\Mod(\sk C)$.
  
  A morphism of sketches $f:\sk C\rightarrow \sk D$ induces, by restriction of the pullback of presheaves, a limit-preserving \emph{restriction of cells} functor $f^*:\Mod(\sk D)\rightarrow \Mod(\sk C)$.
  
\end{definition}

\begin{remark}

  Note that our notion of model uses presheaves (left modules) whereas the classical literature on sketches usually prioritises (covariant) functors (right modules).
  Hence, our models are most closely analogous to what classically would have been called models of the opposite (limit) sketch.

\end{remark}

\begin{definition}[Morita equivalence]
\label{sketch/morita-equivalence/definition}

  A morphism $\sk C\rightarrow \sk D$ of $\Universe$-small sketches is said to be a \emph{$\Universe$-Morita equivalence} if its restriction of cells functor $\Universe\Mod(\sk D) \rightarrow \Universe\Mod(\sk C)$ is an equivalence.
  It is said to be a \emph{Morita equivalence} if it is a $\Universe$-Morita equivalence for some (hence by \cite[Prop.~5.5.7.8]{HTT}, any) $\Universe\gg \sk C,\sk D$.
  An extension $\sk C\rightarrow \sk C\cup R'$ is \emph{Morita trivial} if it is a Morita equivalence of sketches.
  
\end{definition}

\begin{example}[Trivial] \label{sketch/module/example/trivial}

  We have $\Mod(C_\emptyset)=\P(C)$ (cf.~\ref{sketch/example/trivial}).
  
\end{example}

\begin{example}[Cotrivial] \label{sketch/module/example/cotrivial}

  A module for a cotrivial sketch $(C\sep\Cat\coneslice C)$ \eqref{sketch/example/cotrivial} is a final object of $\P(C)$.
  This can be seen already from the relations
  \[
    \emptyset \rightarrow C\slice c;\quad c:C
  \]
  which imply $F(c)\cong \point$ for any $F\in\Mod(C\sep \Cat\coneslice C)$.
  This shows moreover that the extension $\{\emptyset\rightarrow C\slice c\}_{c:C} \subseteq \Cat\coneslice C$ is a Morita equivalence.

\end{example}

\begin{para}[Subcanonical, canonical] \label{sketch/module/subcanonical}

  A sketch is subcanonical \eqref{subcanonical/definition} if and only if the Yoneda embedding of its category of cells has image in $\Mod(\sk C)$.
  Hence we obtain a fully faithful Yoneda embedding $\Yoneda_\sk C:C\rightarrow \Universe\Mod(\sk C)$ for any $\Universe$.
  
  If $C$ is $\Universe$-presentable, then by \cite[Prop.~5.5.2.2]{HTT}, $\Universe$-small models for the $\Universe$-canonical sketch are exactly the representable presheaves; i.e.
  \[
    \Yoneda_{C,\Universe-\colimit}:C\rightarrow\Universe\Mod(C,\Universe-\colimit)
  \] 
  is an equivalence.\footnote{We should be careful around the loose formula $C\cong \Mod(C,\colimit)$ --- the basic rules we are using to resolve implicit universes \eqref{universe/rules} would tell us to pick a universe with respect to which $C$ is \emph{small}, not presentable, and this would make the statement incorrect.}

\end{para}

\begin{proposition}[Slicing over a local presheaf] \label{sketch/module/slice}

  Let $F:\P(C)$ be $R$-cocontinuous. Then for any universe $\Universe$, the natural equivalence $\Universe\P(C\slice F)\cong \Universe\P(C)\slice F$ restricts to an equivalence $\Universe\Mod(\sk C\slice F) \cong \Universe\Mod(\sk C)\slice F$.

\end{proposition}
\begin{proof}

  Let $G:\P(C)\slice F$, and write $G_F$ for the corresponding object of $\P(C\slice F)$ (using Corollary \ref{slice/presheaf}).
  Let $(I\sep \tilde U)$ be a lift of an $R(\sk C)$-diagram $(I\sep U)$ to $C\slice F$, and consider the diagram
  \[
    \begin{tikzcd}
      G_F(\tilde U_\omega) \ar[r] \ar[d] & G(U_\omega) \ar[r] \ar[d] & 
      F(U_\omega) \ar[d, equals] \\
      \lim_{i:I}G_F(\tilde U_i) \ar[r] & \lim_{i:I}G(U_i) \ar[r] &
      \lim_{i:I}F(U_i)
    \end{tikzcd}
  \]
  where the third vertical arrow is invertible by cocontinuity of $F$.
  The left horizontal arrows are the fibres of the right horizontal arrows over the structure morphism of $\tilde U|_{\omega}$, $\tilde U|_I$, respectively.
  Hence, the middle vertical arrow is invertible if and only if the left vertical arrow is.
  In other words, $G_F$ is $R(\sk C\slice F)$-local if and only if $G$ is $R(\sk C)$-local.
  \qedhere
  
\end{proof}

\begin{para}[Extensions]

  If $j:\sk C\rightarrow \sk C\cup R'$ is an extension of sketches, then of course restriction of cells $j^*:\Mod(\sk C\cup R') \rightarrow \Mod(\sk C)$ is fully faithful --- it is the inclusion of the $R'$-continuous presheaves.
  
\end{para}

\begin{para}[Products and modules] \label{sketch/module/product}

  Modules over an external product $\sk C\boxtimes\sk D$ of sketches (\ref{sketch/exterior-product}) are presheaves on $C\times D$ whose restriction to each co-ordinate slice is a module over $\sk C$ or $\sk D$, i.e.~that take relations in either variable to limits.
  It follows that box products of sketches \emph{preserves Morita equivalence}, i.e.~if $\sk C\rightarrow \sk C'$ is a Morita equivalence, then so is $\sk C\boxtimes\sk D\rightarrow\sk C'\boxtimes\sk D$.
  
\end{para}

\begin{example}[Diagram]
\label{sketch/module/example/conical}

  Models of the diagram sketch $(I^\triangleright\sep \{\id_{I^\triangleright}\})$ are simply $I^\op$-indexed limit cones in $\Space$.
  Since every diagram admits a limit cone in $\Space$, the inclusion $I_\emptyset \subset (I^\triangleright)_{\{\id\}}$ is a Morita trivial extension.

\end{example}

\begin{example}[Localisation] \label{sketch/module/example/localization}

  Modules over a localisation sketch \eqref{sketch/example/localization} arising from a relative category $(C\sep W)$ are $W$-local presheaves.

\end{example}


\subsection{Modules as a localisation of presheaves}

The main point of introducing sketches is as presentations of certain localisations of presheaf categories.
Here, we describe these localisations in some detail.

\begin{proposition}[Modules as a Bousfield localisation] \label{sketch/module/bousfield}

  Let $\sk C$ be a sketch on the category $C$.
  The inclusion of $\Mod(\sk C)$ into $\P(C)$ admits a left adjoint
  \[
    L_\sk C:\P(C)=\Mod(C_\emptyset) \rightarrow \Mod(\sk C).
  \]
  which exhibits $\Mod(\sk C)$ as a Bousfield localisation of $\P(C)$ generated by the set of coboundary morphisms of elements of $R(\sk C)$.
  In particular, $\Mod(\sk C)$ is presentable.
  Moreover, $e_F:F\rightarrow L_\sk CF$ can be expressed as a transfinite composition of pushouts of boundary maps associated to elements of $R(\sk C)$.
  
  More generally, if $R'$ is another set of conical diagrams in $C$, then $\Mod(\sk C\cup R')$ is a Bousfield localisation of $\Mod(\sk C)$.
  
\end{proposition}
\begin{proof}

  By criterion \ref{module/criterion/local} of Proposition \ref{sketch/module/criterion}, $\Mod(\sk C)\subseteq\P(C)$ is the set of local objects for a small set of morphisms, so by \pathcite{cis/bousfield-localization/existence} its inclusion admits a left adjoint.
  \qedhere
  
\end{proof}

\begin{definition}[Sketchability]

  An accessible left localisation of $\P(C)$ is said to be \emph{sketchable} if the local objects are exactly the models for some sketch with category of cells $C$.
  
\end{definition}

\begin{corollary}[of Proposition \ref{sketch/module/bousfield}] \label{sketchability}

  An accessible localisation of $\P(C)$ is sketchable if and only if it is generated as a Bousfield localisation by weak equivalences with representable target.
  A sketchable localisation admits a presentation by a $\Universe$-small sketch.
  
\end{corollary}

\begin{para}[Explicit forms of the localiser] \label{sketch/module/localizer/explicit}

  The localisation $L_\sk C$ is generated by what is referred to in \cite{anel2020small} as a `pre-modulator'.
  As such, in addition to the construction of the usual small object argument (which, following \emph{op.~cit}, we dub the $q$-\emph{construction}), $L_\sk C$ can also be described as a transfinite, \emph{convergent}, iteration of a `$k$-construction' or `$+$-construction.'
  
  If $F:\P(C)$, then we will define the \emph{plus construction} (rel.~$R$).
  Denote by $\bar{R}$ the essential image of $R$ in $\P(C)^{\Delta^1}$ and by $\bar{R}\slice F$ the slice over $[F\rightarrow\point]$.
  Note that while each object of $\bar{R}\slice F$ can be expressed as a diagram $(I\sep U)\in R$ with an element of $\lim_{i:I}F(U_i)$, morphisms do not generally lift to maps of diagrams.
  Then $F^+$ is defined to be the pushout of the cospan
  \[
    \begin{tikzcd}
      {|\source|_C} \ar[r] \ar[d] & F \\
      {|\target|_C}
    \end{tikzcd}
  \]
  which is constructed from the natural transformation
  \[ 
    \begin{tikzcd}
      \bar{R}\slice F \ar[r, "\source", bend left] \ar[r, phantom, "\Downarrow"]
      \ar[r, "\target"', bend right] & \P(C)
    \end{tikzcd}
  \]
  and its lift $\source:\bar{R}\slice F\hookrightarrow\P(C)\slice F$.
  After transfinite iteration, this construction converges to $L_\sk CF$ \cite[Thm.~2.4.8]{anel2020small}.
  
  In \S\ref{construction/universal/} it will be useful to have a rewritten form of this colimit.
  First, let us left Kan extend the diagram $\Delta^1\times \bar{R}\slice F$ along $\bar{R}\slice F\rightarrow (\bar{R}\slice F)^\triangleleft$ to obtain a functor
  \[
    \Delta^1\times (\bar{R}\slice F)^\triangleleft \rightarrow \P(C)
  \]
  which maps the $\Delta^1\times\{\text{cone point}\}$ to $\emptyset$.
  Similarly, it is harmless to replace $F$ with the constant $(\bar{R}\slice F)^\triangleleft$-indexed functor with value $F$, as this category is weakly contractible.
  Hence, $F^+$ is also the colimit of a diagram
  \[
    \Lambda^2_2 \times (\bar{R}\slice F)^\triangleleft \rightarrow \P(C)
  \]
  (where $\Lambda^2_2\simeq [1 \leftarrow 0 \rightarrow 2]$ is the walking cospan).
  Pushing out the $\Lambda^2_2$ factor, we find that $F^+$ is a colimit over a diagram indexed by $(\bar{R}\slice F)^\triangleleft$ with initial vertex $F$ and with all other vertices a pushout of an element of $\bar{R}$ along a map from the source to $F$.
  In formulas, we can write this as
  \[
    F^+ \defeq \colim_{W:C,U:\bar{R}(W),U\rightarrow F}W\sqcup_UF
  \]
  in $F\slice \P(C)$.
  In other words,
  \[
    F^+(V) = \colim_{V\rightarrow W,U:\bar{R}(W)} \Map(U,F).
  \]

\end{para}

\begin{example}[Diagram] \label{presheaf/localizer/example/conical}

  Let $I^\triangleright_{\{\id\}}$ be a conical sketch \eqref{sketch/example/conical}.
  The essential image
  \[
    \{\id\} \rightarrow \bar{\{\id\}} \hookrightarrow \P(I^\triangleright)^{\Delta^1}
  \]
  is contractible --- its sole element is modelled by the map of right fibrations
  \[
    \begin{tikzcd}[column sep = tiny]
      I \ar[rr] \ar[dr] && I^\triangleright \ar[dl] \\ & I^\triangleright
    \end{tikzcd}
  \]
  whose automorphism group is trivial because $I\rightarrow I^\triangleright$ is fully faithful, hence monic.
  
  If $F:\P(I^\triangleright)$, the localisation $LF=F^+$ is therefore a pushout
  \[
    \begin{tikzcd}
      F(\omega)\times \omega \ar[r] \ar[d] \ar[dr, phantom, "\ulcorner" near end]
      & F \ar[d] \\
      \left(\lim_{i:I}F(i)\right)\times\omega \ar[r] & F^+
    \end{tikzcd}
  \]
  so that $F^+(\omega)=\lim_{i:I}F(i)$ and $F^+|_I=F|_I$.
  In particular, the plus construction terminates in one step.

\end{example}

\begin{para}[Yoneda functor] \label{sketch/yoneda/definition}

  Composing the standard Yoneda embedding with $(-)^+$ yields a dense functor $\Yoneda_\sk C:\sk C\rightarrow \Mod(\sk C)$ which is a morphism of sketches, but which need not be fully faithful in general.
  In fact, the composite of a Yoneda functor with a Bousfield localisation is fully faithful if and only if representable presheaves are local; so in particular, $\Yoneda_\sk C$ is fully faithful if and only if $\sk C$ is subcanonical.
  In particular, our localised Yoneda functor agrees with the Yoneda embedding for subcanonical sketches defined in \eqref{sketch/module/subcanonical}.
  
\end{para}

\begin{corollary}[Colimit of Yoneda]

  The final object of $\Mod(\sk C)$ is a colimit of $\Yoneda_\sk C$.
  
\end{corollary}

\begin{para}[Factorisation systems on presheaves]
\label{module/factorization-system}

  We have defined our localisations by a generating set of morphisms.
  The small object argument applies, and generates us a weak factorisation system on $\P(C)$ comprising the $R$-\emph{cell complexes} and the $R$-\emph{fibrations}.
  
  In particular, a presheaf $F$ is $R$-\emph{fibrant} if for all $(I\sep U)$ in $R$ the boundary map $F(U_\omega)\rightarrow\lim_{i:I}F(U_i)$ admits a section. 
  All $R$-continuous presheaves are $R$-fibrant.
  
\end{para}


\subsection{Modules as a localisation of diagrams} \label{sketch/localizer/}

Combining Proposition \ref{sketch/module/bousfield} with \ref{diagram/localizer} we obtain a left localisation of $\Cat\slice C$ satisfying a generalised form of the axioms for a fundamental localiser.

\begin{definition}[$R$-equivalence]

  A lax map of diagrams (\ref{diagram/lax-map}) in $C$ is said to be an $R$-\emph{equivalence} if it induces an isomorphism in $\P_R(C)$.
  
\end{definition}

\begin{corollary}[Modules as a localisation of diagrams]

  $\Mod(\sk C)$ is a Bousfield localisation of $\Cat\slice C$ at the set of strict $R$-equivalences.
  
\end{corollary}

\begin{corollary}[Localisers from sketches] \label{sketch/localizer/properties}

  Let $W_\sk C\subseteq (\Cat\slice C)^{\Delta^1}$ be the set of $R(\sk C)$-equivalences of diagrams.
  Then:
  \begin{enumerate}
  
    \item \label{sketch/localizer/properties/saturation}
      $W_\sk C$ is strongly saturated.
      
    \item \label{sketch/localizer/properties/cells}
      If $I$ has a final object $\omega$, then 
      \[
        [\omega:\point\rightarrow I\stackrel{U}{\rightarrow} C] \in W_\sk C
      \]
      for all $U:\Fun(I,C)$.
      
    \item \label{sketch/localizer/properties/quillen}
      If $I\stackrel{h}{\rightarrow} J\rightarrow K\rightarrow C$ is such that
      \[
        [h\slice k:I\slice k\rightarrow J\slice k\rightarrow C]\in W_\sk C
      \]
      for all $k:K$, then $h\in W_\sk C$. \hfill \emph{Quillen A}
  \end{enumerate}
  Moreover, $W_\sk C$ is \emph{of cellular generation}: it is minimal among strongly saturated sets of maps of diagrams containing $W_\sk C\cap(\Cat\coneslice C)$, cofinal maps of diagrams, and satisfying Quillen A.
  
  Conversely, any set of maps $W\subseteq (\Cat\slice C)^{\Delta^1}$ satisfying the preceding properties is the set of weak equivalences for a sketch on $C$.
  
\end{corollary}
\begin{proof}

  Property \ref{sketch/localizer/properties/saturation} is obvious, \ref{sketch/localizer/properties/cells} follows from Corollary \ref{diagram/localizer} and the fact that $W_C\subseteq W_\sk C$, and \ref{sketch/localizer/properties/quillen} follows from Lemma \ref{shape/quillen-a}.
  
  For the remaining statements, we introduce a lemma:
  
\begin{lemma} \label{sketch/localizer/image}

  Let $W\subseteq(\Cat\slice C)^{\Delta^1}$ be strongly saturated and contain the cofinal maps of diagrams.
  Then $W$ is the preimage under $|-|_C$ of its image in $\P(C)^{\Delta^1}$.
  
  If in addition, $W$ satisfies \ref{sketch/localizer/properties}-\ref{sketch/localizer/properties/quillen}, then the image of $W$ in $\P(C)^{\Delta^1}$ is a Bousfield class.
  
\end{lemma}
\begin{proof}

  By Corollary \ref{diagram/localizer}, $W$ contains $W_C$ (the set of $|-|_C$-equivalences).
  In particular, for any diagram $U:I\rightarrow C$ the map $I\rightarrow C\slice|U|_C\rightarrow C$ is in $W$.
  The first statement follows immediately.
  If $W$ also satisfies Quillen A, then its image is closed under colimits by Lemma \ref{shape/quillen-a}. \qedhere

\end{proof}
  
  Lemma \ref{sketch/localizer/image} tells us that any $W$ satisfying these three conditions is the preimage of a Bousfield class in $\P(C)$.
  By definition, the image of $W_\sk C$ is generated as such by $R(\sk C)\subseteq W_\sk C\cap(\Cat\coneslice C)$, hence is of cellular generation.
  
  By the same logic, by taking preimage the lattice of sets of arrows in $(\Cat\slice C)^{\Delta^1}$ satisfying all of the above conditions is exactly the lattice of Bousfield classes in $\P(C)^{\Delta^1}$ generated by maps to representables. By Corollary \ref{sketchability}, these are exactly the sketchable localisations of $\P(C)$.
  \qedhere
  
\end{proof}

\begin{remark}[Sketches from localisers]

  By analogy with \cite[Prop.~4.2.4]{cisinski2006prefaisceaux}, it is natural to conjecture that the converse part of Corollary \ref{sketch/localizer/properties} holds with strong saturation replaced with \emph{weak} saturation (i.e.~closure under 2-out-of-3 and retracts), which is often easier to check, and without assuming $W$ contains cofinal maps of diagrams.
  The tricky part of this is to show that such a weakly saturated $W$ contains $W_C$.
  
\end{remark}

\begin{para}[Factorisation systems on diagrams] \label{diagram/factorization-system}

  Translating \eqref{module/factorization-system} to diagrams, $R$-fibrancy of $F$ is the solvability of the lifting problem
  \[
    \begin{tikzcd}
      I \ar[d] \ar[r] & C\slice F \ar[d] \\
      I^\triangleright \ar[r, "U"] \ar[ur, dotted, "\exists" description] & C
    \end{tikzcd}
  \]
  for all $(I\sep U):R$, which resembles a relativised version of the condition for the diagram $C\slice F$ to be \emph{filtered} with the hypothesis ``$I$ finite'' replaced by ``$(I\sep U)\in R$''.
  A similar lifting problem can be used to define $R$-fibrancy/filteredness for arbitrary objects of $\Cat\slice C$.
  
\end{para}


\subsection{Extension theorems}
\label{sketch/extension/}

Here we discuss the extension of morphisms of sketches to cocontinuous functors between module categories.

\begin{para}[Relations on module categories] \label{sketch/module/relations}
 
  We upgrade the category of modules for a sketch $\sk C$ by equipping it with relations the set
  \[
    (\Universe\Cat\slice\sk C)-\colimit \quad \subseteq \quad
    \Universe\Cat\coneslice \Mod(\sk C)
  \]
  of colimit cones over the composite with $\Yoneda_\sk C$ \eqref{sketch/yoneda/definition} of a $\Universe$-small diagram in $C$.
  We resolve the ambiguous notation $\Mod(\sk C)$ and $\mathtt{colimits}(\sk C)$ by choosing $\Universe\gg \sk C$.
  
  By density of $\Yoneda_\sk C$ and Proposition \ref{dense/extension}, we find
  \[
    \Fun(\Mod(\sk C)\sep (E\sep \colimit)) = \Fun^{cc}(\Mod(\sk C)\sep E)
  \]
  is the category of functors that preserve small colimits, for any cocomplete category $E$.

\end{para}

\begin{prop}[Universal property of module categories]
\label{sketch/extension/existence}

  Let $\sk C$ be a sketch, and let $E:\Cat$ be cocomplete with $\sk C\ll E$. 
  Restriction along the localised Yoneda functor $\Yoneda_{\sk C}:\symb(\sk C)\rightarrow \Mod(\sk C)$ induces an equivalence of categories
  \[
    \Fun(\Mod(\mathcal{C}) \sep (E\sep \colimit))  \quad \tilde\rightarrow 
    \quad \Fun(\mathcal{C} \sep (E \sep \colimit)).
  \]

\end{prop}
\begin{proof}

  Combining the universal property of presheaves \pathcite{cis/presheaf/map-out} with that of localization \pathcite{htt/bousfield-localization/map-out}. \qedhere

\end{proof}

\begin{corollary}

  A functor of sketches $f:\sk C\rightarrow \sk D$ is a $\Universe$-Morita equivalence if and only if precomposition with $f$ induces equivalences of spaces
  \[
    \Fun(\sk D, (E\sep\colimit)) \quad\tilde\rightarrow \quad
    \Fun(\sk C, (E\sep \colimit))
  \]
  for all $\Universe$-cocomplete categories $E$.
  That is, Morita equivalences are exactly equivalences of the induced corepresentable functors on cocomplete categories.
  
\end{corollary}

\begin{corollary}

  For any sketch $\sk C$, the maps 
  \[
    \sk C \quad\stackrel{\Yoneda_\sk C}{\longrightarrow} \quad
    (\Mod(\sk C) \sep (\Universe\Cat\slice C)-\colimit) \quad\longrightarrow\quad
    (\Mod(\sk C)\sep \colimit)
  \]
  are Morita equivalences.
  
\end{corollary}

Rephrasing Proposition \ref{dense/extension} to the present context, we recover useful characterisations of the extension:

\begin{proposition}[Characterisation of extensions]
\label{sketch/extension/criterion}

  The following conditions on a functor $f:\Mod(\sk C)\rightarrow D$ are equivalent:
  \begin{enumerate}
  
    \item \label{extension/criterion/kan}
      $f$ is a left Kan extension of its restriction to $C$, and its restriction to $C$ takes elements of $R$ to colimit cones in $D$.
      
    \item \label{extension/criterion/colimit}
      $f$ preserves small colimits.
      
    \item \label{extension/criterion/adjoint}
      (If $D$ is presentable) $f$ admits a right adjoint.
 
  \end{enumerate}
  Moreover, if $D$ admits small colimits then any functor $f:C\rightarrow D$ that maps elements of $R$ to colimit cones admits a left Kan extension to $\Mod(\sk C)$.
  
\end{proposition}
\begin{proof}

  By Proposition \ref{dense/extension} and the adjoint functor theorem. \qedhere

\end{proof}

\begin{para}[Modules as a functor] \label{sketch/modules/as-functor}

  By stability of $R$ under pullback, taking modules over small sketches defines a subfunctor
  \[
    \Sketch^\op \rightarrow \mathrm{Pr^R}\Cat
  \]
  of $\P$, hence by \pathcite{htt/presentable/prl-prr} a functor
  \[
    \Mod:\Sketch \rightarrow \PrL\Cat
  \]
  taking maps of sketches to their unique colimit-preserving extensions.
  
\end{para}

This functor takes exterior products to tensor products:

\begin{proposition} \label{sketch/modules/tensor-product}

  The category of modules of an exterior product of sketches is equivalent to the presentable tensor product of their categories of modules.
  
\end{proposition}
\begin{proof}

  Since box products preserve Morita equivalences \eqref{sketch/module/product}, $\sk C\boxtimes\sk D\rightarrow \Mod(\sk C)\boxtimes\Mod(\sk D)$ is a Morita equivalence, whence
  \begin{align*}
    \Fun^{cc}(\Mod(\sk C\boxtimes\sk D)\sep E_\all) &\cong \Fun(\Mod(\sk C)\boxtimes\Mod(\sk C)\sep E_\all) \\
    &\cong \Fun(\Mod(\sk C)_\all\boxtimes\Mod(\sk D)_\all\sep E_\all)
  \end{align*}
  is the category of bifunctors of $\Mod(\sk C)\times \Mod(\sk D)$ that preserve small colimits in each variable separately.
  That is, it is the category of binary operations in the operad $\Cat_\infty^\otimes$ defined in \cite[Not.~4.8.1.2]{HA}.  
  \qedhere
  
\end{proof}


\subsection{Essential adjunctions} \label{essential/}

A morphism of small sketches $f:\sk C\rightarrow \sk D$ always induces an adjunction $\Mod(\sk C)\rightleftarrows\Mod(\sk D)$, and the adjunction data is obtained by left Kan extension along the Yoneda functor (by \eqref{yoneda/pullback/as-kan-extension}).
Here we discuss criteria for the existence of an extra right adjoint to this adjunction.
This situation arises, for example, in topos theory, where a further left adjoint to the push-pull adjunction of a geometric morphism computes some kind of relative homotopy type.

\begin{definition}[Essential functor] \label{essential/definition}

  A left adjoint functor between presentable categories is said to be \emph{left essential} if its right adjoint itself has a right adjoint.
  By the adjoint functor theorem, this is equivalent to admitting a colimit-preserving right adjoint.
  
\end{definition}

\begin{example}

  Left essential functors $\Space\rightarrow \P(C)$ are one-to-one with presheaves which can be expressed as retracts of representable presheaves (i.e.~what Lurie calls `completely compact' objects \cite[\S5.1.6]{HTT}).
  
\end{example}

\begin{example}

  Let $f^*:T\rightleftarrows T':f_*$ be an essential geometric morphism of topoi. Then the left adjoint $f_!$ to pullback is a left essential functor (while $f_*$ is what we would call \emph{right} essential if we cared to discuss such things).
  
\end{example}

\begin{proposition}[Criteria for extra adjoints] \label{essential/criteria}

  Let $f:\sk C\rightarrow\sk D$ be a functor of sketches.
  The following conditions are equivalent:
  \begin{enumerate}
  
    \item \label{essential/criteria/beck-chevalley}
      The localisation square
      \[
        \begin{tikzcd}
          \P(C) \ar[r, "f^*"] \ar[d, "L_\sk C"] & \P(D) \ar[d, "L_\sk D"] \\
          \Mod(\sk C) \ar[r, "f^*"] & \Mod(\sk D)
        \end{tikzcd}
      \]
      is Beck-Chevalley.
      
    \item \label{essential/criteria/descent}
      The composite $L_\sk C\circ f^*:\P(D)\rightarrow\Mod(\sk C)$ inverts $R(\sk D)$-equivalences.
      
    \item \label{essential/criteria/relation}
      The composite $L_\sk C\circ f^\dagger:D\rightarrow \Mod(\sk C)$ takes $R(\sk D)$-cones to colimits.
    
    \item \label{essential/criteria/extra-adjoint}
      The restriction of cells functor $f^*:\Mod(\sk D)\rightarrow\Mod(\sk C)$ preserves colimits.
      
    \item \label{essential/criteria/essential}
      The extension $f_!:\Mod(\sk C)\rightarrow\Mod(\sk D)$ is left essential.
      
  \end{enumerate}
  
\end{proposition}
\begin{proof}

  \emph{\ref{essential/criteria/beck-chevalley}$\Leftrightarrow$\ref{essential/criteria/descent}} by the universal property of the localisation $L_\sk D$.
  \emph{\ref{essential/criteria/descent}$\Leftrightarrow$\ref{essential/criteria/relation}} because $L_\sk C\circ f^*$ preserves colimits and $R$-cones generate $R$-equivalences as a Bousfield class.
  \emph{\ref{essential/criteria/relation}$\Leftrightarrow$\ref{essential/criteria/extra-adjoint}} by Proposition \ref{sketch/extension/criterion}, because $f^*$ is a left Kan extension of $L_\sk C\circ f^\dagger$ along $\Yoneda_\sk D$.
  \emph{\ref{essential/criteria/extra-adjoint}$\Leftrightarrow$\ref{essential/criteria/essential}} by the adjoint functor theorem.
  \qedhere
  
\end{proof}

\begin{definition}[Essential morphism of sketches] \label{essential/sketch/definition}

  A functor $f:\sk C\rightarrow \sk D$ between sketches is said to be \emph{(left) essential} if it satisfies the equivalent conditions of Propositio \ref{essential/criteria}.

\end{definition}

\begin{corollary} [Essential functors and indexed localisers] \label{essential/localizer}

  If $f:\sk C\rightarrow \sk D$ is an essential functor of sketches, then the restriction of cells and comma category construction fit into commuting squares
  \[
    \begin{tikzcd}
      \Cat\slice D \ar[r, "C\downarrow_D(-)"] \ar[d] & \Cat\slice C \ar[d] \\
      \Mod(\sk D) \ar[r, "f^*"] & \Mod(\sk C)
    \end{tikzcd}
  \]
  In particular, $C\downarrow_DW_\sk D \subseteq W_\sk C$.

\end{corollary}

\begin{corollary}

  Morita equivalences are essential.
  
\end{corollary}

\begin{proposition}[Fully faithful essential functors] \label{essential/fully-faithful}

  Let $f:\sk C\rightarrow \sk D$ be an essential morphism of sketches.
  Then $f_!:\Mod(\sk C)\rightarrow\Mod(\sk D)$ is fully faithful if and only if $\Yoneda_\sk C\rightarrow f^\dagger f$ is invertible.
  
\end{proposition}
\begin{proof}

  Since $f$ is essential, $f^*f_!:\Mod(\sk C)\rightarrow \Mod(\sk C)$ preserves colimits, whence the invertibility of the unit $\id_{\Mod(\sk C)}\rightarrow f^*f_!$ can be checked after restricting along the dense functor $\Yoneda_\sk C$. \qedhere

\end{proof}

\begin{example}[Essential, subcanonical extension]

  Let $j:\sk C\hookrightarrow \sk C\cup R'$ be an extension which is essential as a morphism of sketches.
  Suppose that $\sk C$ is subcanonical, whence $j^\dagger = \Yoneda_\sk C = \Yoneda_C$.
  In particular, $j^\dagger$ is dense, whence its left Kan extension $j^\dagger:\Mod(\sk C\cup R')\rightarrow\Mod(\sk C)$ is essentially surjective.
  Since it is also fully faithful, it is an equivalence; i.e.~$j$ is a Morita equivalence.
  
\end{example}


\subsection{Saturation} \label{saturation/}

Morita equivalence defines an equivalence relation on the lattice of sketches on a fixed category $C$.
Here we make general statements about the maximal, or \emph{saturated}, representatives for each class.
Unfortunately, these saturated sets of diagrams are large and complicated, and it seems to be difficult to identify practical necessary and sufficient criteria for membership thereof.

\begin{lemma}[Acyclicity] \label{acyclic-cone/criteria}

  Let $U:I\rightarrow C\slice U_\omega$ be a conical diagram in a sketch $\sk C$. The following are equivalent:
  \begin{enumerate}
  
    \item \label{acyclic-cone/criteria/continuity}
      For every $F:\Mod(\sk C)$, $F\circ U:(I^\triangleright)^\op\rightarrow \Space$ is a limit cone. 
      
    \item \label{acyclic-cone/criteria/equivalence}
      The coboundary map $d_U:|U|_C\rightarrow U_\omega$ in $\P(C)$ is an $R$-equivalence.
    
    \item \label{acyclic-cone/criteria/colimit}
      The composite $\Yoneda_\sk C\circ U$ is a colimit cone in $\Mod(\sk C)$.
  \end{enumerate}

\end{lemma}
\begin{proof}

  We have $\lim_{i:I}F(U_i)=\lim_{i:I}\Map(U_i,F) = \Map(|U|_C,F)$ in $\P(C)$.
  Hence \eqref{acyclic-cone/criteria/continuity} holds if and only if every $R$-continuous presheaf is $d_U$-local, i.e.~\eqref{acyclic-cone/criteria/equivalence}.
  The latter equivalence \eqref{acyclic-cone/criteria/equivalence}$\Leftrightarrow$\eqref{acyclic-cone/criteria/colimit} follows straight from definitions. \qedhere
  
\end{proof}

\begin{definition}[$R$-acyclic cone] \label{acyclic-cone/definition}

  A conical diagram in a colimit-sketch $\sk C$ is said to be an \emph{acyclic cone}, or \emph{$R(\sk C)$-acyclic cone}, if it satisfies the conditions of Lemma \ref{acyclic-cone/criteria}.
  The set of small acyclic cones in $\sk C$ is denoted $R-\mathtt{acy}(\sk C)$.

\end{definition}

\begin{lemma} \label{morita-trivial-extension/criterion}

  An extension $\sk C\hookrightarrow \sk C \cup R'$ is Morita-trivial if and only if every element of $R'$ is $R(\sk C)$-acyclic.
  
\end{lemma}
\begin{proof}

  By definition, Morita-triviality of the extension implies $\Mod(\sk C) = \Mod(\sk C\cup R')$ as full subcategories of $\P(C)$, i.e.~if $R(\sk C)$-continuous presheaves take elements of $R'$ to limits (criterion \eqref{acyclic-cone/criteria/continuity} of Lemma \ref{acyclic-cone/criteria}).
  \qedhere

\end{proof}

\begin{lemma}[Cofinal closure] \label{acyclic-cone/cofinal}

  Let $h:(I\sep U) \rightarrow (J\sep V)$ be a cofinal map of diagrams in $C$.
  Then $U$ is $R$-acyclic if and only if $V$ is $R$-acyclic.
  
\end{lemma}
\begin{proof}

  By criterion \ref{acyclic-cone/criteria/continuity} of Lemma \ref{acyclic-cone/criteria}. \qedhere

\end{proof}

\begin{proposition} [Criteria for saturation] \label{saturated/criteria}

  Let $\sk C$ a sketch with small relations.
  The following conditions are equivalent:
  \begin{enumerate}
    \item \label{saturated/criteria/acyclic}
      $R(\sk C)=R-\mathtt{acy}(\sk C)$.
      
    \item \label{saturated/criteria/extension}
      Every Morita trivial extension of $R(\sk C)$ by a set of small conical diagrams is trivial.
      
    \item \label{saturated/criteria/saturation}
      $R(\sk C) = (\Cat\coneslice C)\cap W_\sk C$, where $W_\sk C\subseteq(\Cat\slice C)^{\Delta^1}$ is the set of $R(\sk C)$-equivalences.
      
    \item \label{saturated/criteria/cofinally-closed}
      If $h:(I\sep U) \rightarrow (J\sep V)$ is a cofinal map of small diagrams in $C$, then 
      \[
        (I\sep U)\in R(\sk C) \quad \Leftrightarrow \quad (J\sep V)\in R(\sk C);
      \]
      and if $S\subseteq R(\sk C)$ denotes the set of relations $(I\sep U)\in R(\sk C)$ such that $U$ is a right fibration, then 
      \[
        S = \langle S \rangle \cap (\RFib(C)\downarrow C)
      \]
      where $\langle S\rangle\subseteq\RFib(C)^{\Delta^1}$ is the Bousfield class generated by $S$ and $(\RFib(C)\downarrow C)\subset\RFib(C)^{\Delta^1}$ is the set of right fibrations over a slice of $C$.
      
  \end{enumerate}
  Moreover, for any sketch $\sk C$ the extension $\sk C \cup R-\mathtt{acy}(\sk C)$ is Morita-trivial, and $R-\mathtt{acy}(\sk C)$ satisfies the preceding conditions.
  
\end{proposition}
\begin{proof}

  \begin{labelitems}
  
    \item[\ref{saturated/criteria/acyclic}$\Leftrightarrow$\ref{saturated/criteria/extension}]
      By Lemma \ref{morita-trivial-extension/criterion}.
      
    \item[\ref{saturated/criteria/acyclic}$\Leftrightarrow$\ref{saturated/criteria/saturation}]
      By criterion \ref{acyclic-cone/criteria/equivalence} of Lemma \ref{acyclic-cone/criteria}.
      
    \item[\ref{saturated/criteria/acyclic}$\Leftrightarrow$\ref{saturated/criteria/cofinally-closed}]
      The first condition is equivalent to the statement that 
      \[
        (I\sep U)\in R \quad \Leftrightarrow \quad (C\slice |U|_C\rightarrow C\slice |U|_\sk C )\in R
      \]
      since any $(J\sep V)$ cofinal with $(I\sep U)$ will have isomorphic classifying presheaf.
      Thus, it is equivalent to the statement that $W_\sk C$ is the preimage of $S$ under the reflection $\Cat\slice C\rightarrow \RFib(C)$.
      Therefore, under these conditions $W_\sk C$ is strongly saturated if and only if $S$ is.
      
    \item[Morita triviality]
      By criterion \ref{acyclic-cone/criteria/continuity} of Lemma \ref{acyclic-cone/criteria}, a presheaf is $R(\sk C)$-continuous if and only if it is $R-\mathtt{acy}(\sk C)$-continuous.
      In particular, the concept of $R$-acyclicity is invariant under Morita-trivial extensions.
      Since $R$-acyclicity depends only on the category of cells and the Morita equivalence class, we are done. \qedhere

  \end{labelitems}
  
\end{proof}

\begin{definition}[Saturation]
\label{saturated/definition}

  A set of relations $R$ on a category $C$ satisfying the conditions of the Proposition is said to be $\Universe$-\emph{saturated}.
  Note that a saturated set of relations on a small, nonempty category is never small.

\end{definition}

\begin{para}[Absolute colimits and $\emptyset$-acyclicity]
\label{acyclic/example/absolute}

  Let us introduce the term \emph{aspherical} for $\emptyset$-acyclicity of diagrams.
  A conical diagram $(I\sep U)$ in $C$ is aspherical, hence $R$-acyclic for any $R$, if and only if its composite with Yoneda is a colimit in $\P(C)$.
  This implies that the colimit $U$ is preserved by \emph{any} functor, i.e.~it is an \emph{absolute} colimit: indeed, if $f:C\rightarrow D$ is any functor, then applying the extension theorem \ref{sketch/extension/existence} yields a colimit-preserving functor $f_!:\P(C)\rightarrow\P(D)$, whence $f_!\circ\Yoneda_C\circ U = \Yoneda_D\circ f\circ U$ is a colimit cone, whence $f\circ U$ is a colimit cone because $\Yoneda_D$ is fully faithful.  
  
  Examples of absolute colimits include colimits over:
  \begin{itemize}
    \item Any diagram indexed by a category with terminal object;
    \item A constant diagram indexed by a weakly contractible category (these are exactly the aspherical diagrams with codomain the one-point category);
    \item A diagram indexed by the walking idempotent $\Idem$ \pathcite{htt/idempotent/walking}.
    \item More generally, a diagram indexed by a category that admits a final idempotent (see \eqref{construction/universal/idempotent}).
  \end{itemize}  
  A right fibration is aspherical if and only if its index category has a final object.
  Hence, by Proposition \ref{saturated/criteria}, a diagram is aspherical if and only if it factors through a cofinal functor into a category with a final object.
  
  Since replacing a set of diagrams with its saturation is a monotone function, \emph{every} saturated set of diagrams on $C$ contains $\langle\emptyset\rangle_{\Cat\slice C}$.
  
\end{para}

\begin{example}[$\mathtt{fincoprod}$-acyclicity]

  Let $C$ be a category with finite coproducts. Let $\mathtt{fincoprod}(C)$ be the set of finite coproduct diagrams in $C$.
  A conical diagram is $\mathtt{fincoprod}$-acyclic if it's a colimit cone over a finite disjoint union of aspherical diagrams.
  But this condition isn't necessary: for example, if $X\sqcup Y\simeq Z$ in $C$ then the conical diagram
  
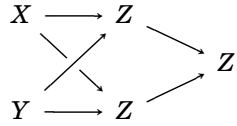
\[
    \begin{tikzcd}[row sep = tiny]
      X \ar[r] \ar[ddr] & Z \ar[dr] \\
      && Z \\
      Y \ar[r] \ar[uur, crossing over] & Z \ar[ur]
    \end{tikzcd}
  \]
  (where both squares are commutative) is $\mathtt{fincoprod}$-acyclic.
  
\end{example}


\subsection{Sketches on a point} \label{sketch/on-point/}

Sketches on a point present localisations of $\Space$.
Not surprisingly, the localisations that arise in this way are a known class have appeared in the literature under various guises.

\begin{proposition} \label{sketch/on-point/classification}

  The following posets are canonically equivalent:
  \begin{enumerate}
  
    \item \label{sketch/on-point/classification/sketch}
      Sketchable localisations of $\Space$ (considered as presheaves on a point);
      
    \item \label{sketch/on-point/classification/nullity}
      Nullity classes of spaces;
      
    \item \label{sketch/on-point/classification/localizer}
      Proper fundamental localisers on $\Cat_1$.
      
  \end{enumerate}
  A set $R\subseteq \Cat$ of diagrams on a point is saturated if and only if there is some nullity class $\bar{R}$ of spaces such that $I\in R$ if and only if $|I|\in\bar{R}$.
  
\end{proposition}
\begin{proof}

  The equivalence \ref{sketch/on-point/classification/localizer}$\Leftrightarrow$\ref{sketch/on-point/classification/sketch} is from \cite[Thm.~6.1.11]{cisinski2006prefaisceaux}. 
  For the equivalence with \ref{sketch/on-point/classification/nullity} (in the sense of \cite[\S A.4]{farjoun2006cellular}), observe that a space is local with respect to a conical diagram $U:I^\triangleright \rightarrow \point$ if and only if is is local with respect to the groupoid completion $|I|$ of the index category, so that $R$-cocontinuous presheaves are exactly the $W$-null spaces for some suitable set of spaces $\{W\}$.\qedhere
  
\end{proof}

By pullback, these sketches induce natural sketches on any category; hence the phenomena appearing here are universal in sketch theory.
Extensive discussion of the actual structure of this lattice is to be found in Farjoun's monograph \cite{farjoun2006cellular}.
Here, I only pause to recall a few important examples.

\begin{example}[Acyclic spaces and Quillen plus construction] \label{sketch/example/acyclic-spaces}

  A motivating example for the study of nullification functors is the sketch on a point whose relations are the (constant conical diagrams on) \emph{acyclic spaces} \cite{hausmann1979acyclic}.
  Farjoun's acyclic cover functor $X\mapsto AX$ takes a space $X$ to the final object of $\bar{R}\slice X$, where $\bar{R}\subset\Space$ is the full subcategory spanned by acyclic spaces \cite[Thm.~2.1]{dror1972acyclic}.
  By \cite[Thm.~4.1, 4.2]{hausmann1979acyclic}, the sequence
  \[
    AX \rightarrow X \rightarrow X^+_\mathrm{Q}
  \]
  is a fibre-cofibre sequence, where $X^+_\mathrm{Q}$ is Quillen's plus construction.
  In other words, we have a pushout
  \[
    \begin{tikzcd}
      \colim_{U:\bar{R}/X}U \ar[r, equals] \ar[d] & 
      AX \ar[r] \ar[d] \ar[dr, phantom, "\ulcorner" near end] & X \ar[d] \\
      {|\bar{R}\slice X|} \ar[r, equals] & * \ar[r] & X^+_\mathrm{Q}.
    \end{tikzcd}
  \]
  So Quillen's plus construction is a special case of the plus construction $X^+=L_{\Space,R}X$ of \eqref{sketch/module/localizer/explicit}.

\end{example}

\begin{example}[Non-sketchable localisation of $\Space$]

  Localisations of $\Space$ that do not arise as sketches on a point are abundant. 
  For example, the set of $H\Z$-equivalences gives such a localisation: the acyclic objects the same as in Example \ref{sketch/example/acyclic-spaces}, but there are spaces with Abelian fundamental group (but with $\pi_1$ acting non-trivially on higher homotopy groups) that are not $H\Z$-local \cite{196604}.
  However, every such localisation $L$ induces a sketch whose relations are the $L$-trivial spaces, and hence another localisation (which may be finer than $L$).
  
\end{example}

\begin{example}[Truncations]

  Let $R_n\subseteq\Space$ be the set $\{S^k\}_{k>n}$ of spheres of dimension greater than $n$. 
  Then $R_n$-local spaces are $n$-truncated spaces, and $\bar{R}_n\subseteq\Space$ is the category of $n$-connected spaces.
  More generally, for any category $C$ the $R_n$-local presheaves are exactly the $n$-\emph{truncated} presheaves \cite[Def.\ 5.5.6.1]{HTT}, i.e.~the presheaves with values in $n$-truncated spaces.
  
\end{example}

To illustrate the effect of switching between $1$- and $\infty$-category theory, we end this section with a discussion of Bousfield localisations of $\Set$, that is, Bousfield localisations of $\Space$ containing the one sketched by $R=\{S^n\}_{n>0}$.
  
\begin{lemma}[Localisations of $\Set$]

  If $L:\Set\rightarrow\Set$ is a Bousfield localisation, then either $L$ is an equivalence or $L$ inverts all surjective maps.

\end{lemma}
\begin{proof}
  
  Let $f:A\rightarrow B$ be a non-bijective $L$-equivalence.
  If $A=\emptyset$ then $\emptyset\rightarrow\point$ is a retract of $f$, hence an $L$-equivalence; hence all maps are $L$-equivalences.
  So we may assume $A$ is nonempty; we will show that $\mathbf 2\rightarrow\point$ is an $L$-equivalence.
  If $f$ is injective but not surjective, then it admits a retract onto $\mathbf 1\rightarrow\mathbf 2$ and hence $\mathbf 2\rightarrow\point$ is an $L$-equivalence
  Otherwise, suppose that $f$ is not injective; then there are $a\neq a'\in A$ such that $f(a)=f(a')$. 
  We can define a retraction of $A$ onto $\{a,a'\}$ by sending every element of $A\setminus\{a,a'\}$ to $a$.
  It follows that $\{a,a'\}\rightarrow \point$ is an $L$-equivalence.
  
  Now, by stability under coproducts, any map $\mathbf{n}\rightarrow\mathbf{n-1}$ is an $L$-equivalence, and by composition, so is any surjective map between finite sets.
  By stability under filtered colimits, all surjective maps with target a point are $L$-equivalences.
  Finally, the result follows by stability under coproducts.
  \qedhere
  
\end{proof}

So any nontrivial Bousfield localisation of $\Set$ factors through $\tau_{-1}:\Set\rightarrow \Delta^1$, hence is equivalent either to $\tau_{-1}$ or to the cotrivial localisation $\Set\rightarrow \point$.
In particular, both are sketchable.

However, there are still non-sketchable localisations within the domain of 1-topos theory:

\begin{example}[Non-sketchable localisations of $G-\Set$]

  The category of $H$-trivial $G$-sets, where $G$ is a discrete group, is limit-closed but not sketchable in $\Set$ nor $\Space$. Indeed, there are no nontrivial resolutions of the free cyclic $G$-set: $G-\Space\slice G\cong\Space$.

\end{example}


\section{Constructions} \label{construction/}

  In this section we turn to the task of describing presentations of modules over sketches in terms of diagrams of restricted shape.
  In \S\ref{construction/extension/}, we prove a recognition lemma for left Kan extensions, and derive universal properties and a composition law for categories of constructible modules.
  Then, using some facts about density and rectification recalled in \S\ref{construction/rectification/}, we study in \S\ref{construction/universal/} the question of idempotence of cocompletion under certain sets of colimits defined universally for all categories (or a family of categories).
  
\subsection{Diagrams} \label{construction/diagram/}

If $\sk C$ is a sketch, then a ($\Universe$-)\emph{diagram in $\sk C$} is defined to be a ($\Universe$-)diagram in $\Sigma(\sk C)$.
The remarks of \eqref{conical-diagram/set} and \eqref{conical-diagram/functoriality} apply \emph{mutatis mutandis} to sets of (non-conical) diagrams.

\begin{para}[Diagrams and conical diagrams]

  A conical diagram $U:J\rightarrow C\slice \omega$ has an \emph{underlying diagram} $U|_{J}=\forget^\omega(U)$ obtained by composition with the projection $\source:C\slice \omega\rightarrow C$.
  Or, interpreting $U$ as a functor $J^\triangleright\rightarrow C$, $\forget^\omega(U)$ is forgetting the value on the cone point.
  Or, interpreting $U$ as a morphism in $\Cat\slice C$ with constant target, $\forget^\omega(U)$ is simply the source.
  We thereby obtain a map
  \[
    \forget^\omega:\Object(\Cat \coneslice C) \rightarrow \Object(\Cat\slice C).
  \]
  \begin{itemize}
    
    \item
      Given a set of conical diagrams $R$, the set $\forget^\omega(R)$ of underlying diagrams is denoted $R^\circ$.
      
    \item
      Given a set of diagrams $K$ in $C$, all of whose members admit a colimit in $C$, we write $K-\mathtt{colimit}$ for the set of conical diagrams which are colimit cones of elements of $K$.
      This generalises the notation $(\Universe\Cat\slice \sk C)-\colimit$ from \eqref{sketch/module/relations}.
      
  \end{itemize}
  We have $(R^\circ)-\mathtt{colimit}=R$ if and only if $R$ is subcanonical and $(K-\mathtt{colimit})^\circ = K$ if and only if all elements of $K$ admit colimits.
  
\end{para}

\begin{definition}[Sketch with constructions]

  A \emph{sketch with constructions} in $\Universe$ is a pair $(\sk C\sep K)$ where $\sk C:\Universe\Sketch$ and $K:\Universe\Power(\Universe\Cat\slice C)$ is a set of diagrams that contains all singletons.
  Sketches with constructions form a full subcategory
  \[
    \Sketch^K \subseteq \int_{\sk C:\Sketch}\Power(\Cat\slice \symb(\sk C)).
  \]
  cut out by the singleton condition.
  The forgetful functor $\Sketch^K\rightarrow\Sketch$ is a bi-Cartesian fibration with obvious left and right adjoints.

\end{definition}

\begin{definition}[Constructibly cocomplete] \label{construction/cocomplete/definition}

  A sketch with constructions $(\sk C\sep K)$ is said to be \emph{constructibly cocomplete} if it is subcanonical, every element of $K$ admits a colimit in $\Sigma(\sk C)$, and $K-\colimit\subseteq R(\sk C)$.
  In particular, functors between constructibly cocomplete sketches `preserve $K$-colimits'.
  The full subcategory of $\Sketch^K$ spanned by the constructibly cocomplete sketches is called $\Sketch^\mathrm{cc}$.

\end{definition}

\begin{remark}

  Just as the definition of $\Sketch$ involves fixing three universes, $\Sketch^K$ involves five, and for some questions it is important to allow these to be different.
  
\end{remark}

\begin{definition}[Constructible modules]

  The full subcategory of $\Mod(\sk C)$ spanned by the $K$-constructible modules is denoted $\Mod^K(\sk C)$ or $\Mod(\sk C\sep K)$.
  It is considered as a sketch with:
  \begin{itemize}
    \item
      relations $K-\mathtt{colimits}\subseteq \Cat\coneslice\Mod^K(\sk C)$;
    \item
      constructions $\Yoneda_{\sk C,K,!}(K)\subseteq\Cat\slice\Mod^K(\sk C)$.
  \end{itemize}
  The hypothesis that $K$ contain singletons ensures that the Yoneda functor $\Yoneda_\sk C$ factors through $\Mod(\sk C\sep K)$.
  The category of $K$-constructible presheaves does not depend on the choice of universe $\Universe$ used to define $\Mod(C)$, so long as $\sk C\ll \Universe$ and $K\subseteq \Universe\Cat\slice C$.
  
\end{definition}

\begin{para}[Constructions and slicing] \label{construction/slice}

  Let $\sk C$ be a sketch with constructions and let $F$ be a module.
  If $K\subseteq\Cat\slice \sk C$ is a set of diagrams, define $K\slice F\subseteq \Cat\slice(C\slice F)$ to be the preimage of $K$ under $\source:C\slice F\rightarrow C$.
  Then the square
  \[
    \begin{tikzcd}
      \Mod^{K\slice F}(\sk C\slice F) \ar[r] \ar[d, hook] 
      \ar[dr, phantom, "\lrcorner" near start] & \Mod^K(\sk C) \ar[d, hook] \\
      \Mod(\sk C\slice F) \ar[r] & \Mod(\sk C)
    \end{tikzcd}
  \]
  is a pullback.
  Indeed, if $G\in\Mod(\sk C\slice F)$ is a colimit $G=|U|_\sk C$ in $\Mod(\sk C)$ for some $(I\sep U)\in K$, then postcomposition with $G\rightarrow F$ yields a lift of $(I\sep U)$ to $\sk C\slice F$, whence $G$ is also $K\slice F$-constructible.
  
  It follows that the equivalence $\Mod(\sk C\slice F)\cong \Mod(\sk C)\slice F$ of Proposition \ref{sketch/module/slice} restricts to an equivalence
  \[
    \Mod^K(\sk C)\slice F\cong \Mod^{K\slice F}(\sk C\slice F).
  \]

\end{para}


\subsection{Extension theorems}
\label{construction/extension/}

The main result of this section is an extension lemma \eqref{constructible/extension/uniqueness}, which provides a criterion for an extension along a dense functor to be a left Kan extension and preserve a class of colimits.

\begin{lemma}[Properties of the constructible Yoneda functor] \label{constructible/yoneda/properties}

  Let $(\sk C\sep K)$ be a sketch with constructions.
  Then:
  \begin{enumerate}
    
    \item \label{constructible/modules/structure/density}
      The Yoneda functor $\Yoneda_{\sk C,K}:C\rightarrow \Mod^K(\sk C)$ and the inclusion $j:\Mod^K(\sk C)\hookrightarrow\Mod(\sk C)$ are both dense.
      The former is an essential morphism of sketches.
      
    \item \label{constructible/modules/structure/inclusion-is-extension}
      The inclusion $j$ is uniquely isomorphic to the Yoneda pullback $\Yoneda_{\sk C,K}^\dagger$.
      In particular, it is a left Kan extension of $\Yoneda_{\sk C}$ along $\Yoneda_{\sk C,K}$.
      
  \end{enumerate}

\end{lemma}
\begin{proof}

  By Proposition \ref{dense/properties}-\ref{dense/properties/fully-faithful}. \qedhere

\end{proof}

\begin{proposition}[Uniqueness of constructible extensions] \label{constructible/extension/uniqueness}

  Let $(\sk C\sep K)$ be a sketch with constructions, $g:\Mod^K(\sk C)\rightarrow (D\sep\colimit)$ a functor of sketches into a cocomplete category (i.e.~which preserves all $R(\sk C)$-colimits and $K$-colimits).
  Then $g$ is a left Kan extension of its restriction to $\sk C$, it preserves all colimits of diagrams in $C$ which exist in $\Mod^K(\sk C)$, and is the unique extension of $g$ that preserves $K$-colimits.
  
\end{proposition}
\begin{proof}

  Denote by $j:\Mod^K(\sk C)\hookrightarrow\Mod(\sk C)$ the dense inclusion.
  There is a unique map 
  \[
    \Yoneda_{\sk C,!}(g|C) \rightarrow j_!(g)
  \]
  extending the identity on $g|C$, and this is an equivalence on $K$-constructible objects because both functors preserve $K$-colimits.
  It follows that $g$ preserves all $K$-constructible colimits, and uniqueness follows from that of Kan extensions. \qedhere
  
\end{proof}

\begin{remark}

  To see how powerful the conclusion of Proposition \ref{constructible/extension/uniqueness} is, pause to observe how easily it fails without the sketch data presenting $\Mod^K(\sk C)$, i.e.~in the context of a general dense functor.
  For example, it is easy enough to construct examples of dense maps $C\rightarrow D$ of posets where a functor $D\rightarrow E$ preserves a colimit $X=\colim U$ of a diagram $U:I\rightarrow C$, but does not preserve colimits over other diagrams with apex isomorphic to $X$ (even if we assume $X\not\in C$).
  A concrete example is the inclusion of $C\subseteq D\defeq\Power(3)$ of the set of proper subsets of a 3-element set: a functor on $D$ preserving $2\sqcup 1 = 3$ needn't preserve $1\sqcup 1\sqcup 1 = 3$.
  
\end{remark}

\begin{para}

  The commuting square
  \[
    \begin{tikzcd}
      C \ar[r, "\Yoneda_{\sk C,K}"] \ar[d] & \Mod^K(\sk C) 
      \ar[d, "\Yoneda_{\Mod^K(\sk C)}"] \\
      \Mod(\sk C) \ar[r, "\Yoneda_{\sk C,K,!}"'] & \Mod(\Mod^K(\sk C)).
    \end{tikzcd}
  \]
  obtained from functoriality of $\Mod$ yields, by left Kan extension along $\Yoneda_{\sk C,K}$, a triangle
  \[
    \begin{tikzcd}
       & \Mod^K(\sk C) 
      \ar[d, "\Yoneda_{\Mod^K(\sk C)}"] 
      \ar[dl, "j=\Yoneda_{\sk C,K}^\dagger" description] \\
      \Mod(\sk C) \ar[r, "\Yoneda_{\sk C,K,!}"'] & \Mod(\Mod^K(\sk C)).
    \end{tikzcd}
  \]
  Putting this together with the strict left Kan extension $\Yoneda_\sk C\cong j\Yoneda_{\sk C,K}$, we obtain a diagonal lift of the preceding square:
  \[
    \begin{tikzcd}
      C \ar[r, "\Yoneda_{\sk C,K}"] \ar[d] & \Mod^K(\sk C) 
      \ar[d, "\Yoneda_{\Mod^K(\sk C)}"] 
      \ar[dl, "j=\Yoneda_{\sk C,K}^\dagger" description] \\
      \Mod(\sk C) \ar[r, "\Yoneda_{\sk C,K,!}"'] & \Mod(\Mod^K(\sk C)).
    \end{tikzcd}
  \]
  
\end{para}

We will now show that the lower horizontal arrow is an equivalence of categories.

\begin{proposition}[Transitivity]
\label{constructible/transitive}

  Let $\sk C$ be a colimit sketch on a category $C$, $K$ a set of constructions on $\sk C$.
  The Yoneda functor $\Yoneda_{\sk C,K}:\sk C\rightarrow \Mod^K(\sk C)$ is a Morita equivalence.
  That is, there is a unique equivalence
  \[
    \Mod(\sk C) \cong \Mod(\Mod^K(\sk C))
  \]
  compatible with the Yoneda functor of $\sk C$.

\end{proposition}
\begin{proof}

  Since by Lemma \ref{constructible/yoneda/properties} $\Yoneda_{\sk C,K}$ is an essential functor of sketches and 
  \[
    \Yoneda_{\sk C,K}^\dagger\Yoneda_{\sk C,K} = j\Yoneda_{\sk C,K} = \Yoneda_\sk C,
  \]    
  its extension $\Yoneda_{\sk C,K,!}:\Mod(\sk C)\rightarrow \Mod(\Mod^K(\sk C))$ is fully faithful by Proposition \ref{essential/fully-faithful}.
  On the other hand, by Proposition \ref{dense/properties}-\ref{dense/properties/extension}, $\Yoneda_{\sk C,K,!}$ is dense and hence, being also a left adjoint, a Bousfield localisation by \ref{dense/properties}-\ref{dense/properties/left-adjoint}. \qedhere
  
\end{proof}

\begin{corollary}[Universal property of categories of constructible modules]
\label{constructible/universal}

  Let $(\sk C\sep K)$ be a sketch with constructions. The Yoneda functor $\Yoneda_\sk C:\sk C\rightarrow \Mod(\sk C)$ induces an equivalence of categories
  \[
    \Fun((\Mod^K(\sk C)\sep K_\sk C)\sep (\sk D\sep K_\sk D)) \cong 
      \Fun( (\sk C\sep K_\sk C)\sep (\sk D\sep K_\sk D))
  \]
  for any constructibly cocomplete sketch $(\sk D\sep K_\sk D)$.

\end{corollary}
\begin{proof}

  Embed $\sk D$ in $\Mod(\sk D)$ (using that $\sk D$ is subcanonical), and identify
  \[
    \begin{tikzcd}
      \Fun(\Mod^K(\mathcal{C})\sep \mathcal{D}) \ar[r, hook] \ar[ddd, equals] &
        \Fun( \Mod^K(\mathcal{C})\sep \Mod(\mathcal{D}) ) \ar[d, equals, "\text{extension}" description] \\
      &
        \Fun ( \Mod(\Mod^K(\sk C)) \sep \Mod(\sk D) ) \ar[d, equals, "\text{transitivity}" description] \\
      &
        \Fun ( \Mod(\mathcal{C}) \sep \Mod(\mathcal{D}) ) \ar[d, equals, "\text{extension}" description] \\
      \Fun^K( \mathcal{C}\sep \mathcal{D}) \ar[r, hook] &
        \Fun ( \sk C \sep \Mod(\sk D) )
    \end{tikzcd}
  \]
  noting that this diagram commutes because the identification $\Mod(\Mod^K(\sk C))\cong \Mod(\sk C)$ commutes with the inclusion of $\Mod^K(\sk C)$. \qedhere
  
\end{proof}

\begin{corollary}

  The full subcategory $\Sketch^{K,cc}$ of $\Sketch^K$ spanned by the constructibly cocomplete sketches is reflective.
  
\end{corollary}
\begin{proof}

  By Corollary \ref{constructible/universal} and \cite[Prop.~6.1.11]{HCHA}. \qedhere
  
\end{proof}

\begin{remark}

  Given a sketch with constructions $(\sk C\sep K)$, we obtain by extension of $(C_\emptyset\sep K)\rightarrow(\sk C\sep K)$ an essentially surjective, $K$-colimit preserving functor
  \[
    \P^K(C)\rightarrow\Mod^K(\sk C).
  \]
  We might ask if this functor is itself a localisation, or even a Bousfield localisation.
  This boils down to a tricky question of whether zig-zags of weak equivalences ending at $K$-constructible presheaves can be transformed to zig-zags of weak equivalences \emph{through} $K$-constructible presheaves.
  I don't know any useful general criteria to answer this.
  
\end{remark}

\begin{para}[As presentations for dense functors]

  Let $\mathbf{Dense}$ be the full subcategory of $\Cat^{\Delta^1}$ spanned by the dense functors (see \S\ref{dense/}).
  One can realise $\Mod^K$ as a functor
  \[
    \Sketch^K \rightarrow \Dense
  \]
  by picking it out as a subfunctor of the restriction of the functor
  \[
    [\Yoneda:\Sigma\rightarrow\Mod]:\Sketch\rightarrow \Universe\Cat^{\Delta^1}
  \]
  constructed in \ref{sketch/modules/as-functor}. Here $\Universe$ is a universe with respect to which $\PrL\Cat$ is small.
  
\end{para}


\subsection{Rectification} \label{construction/rectification/}

This paragraph is about constructing weak inverses to Morita equivalences between sketches with different categories of cells.

\begin{definition}[Constructible Morita equivalence] \label{constructible/morita/definition}

  A morphism $f:(\sk C\sep K_\sk C)\rightarrow (\sk D\sep K_\sk D)$ of sketches with constructions is a \emph{constructible equivalence} if its extension $\Mod^{K_\sk C}(\sk C)\rightarrow\Mod^{K_\sk D}(\sk D)$ is an equivalence of categories.
  It is a \emph{constructible Morita equivalence} if it is a constructible equivalence and the following, equivalent conditions are satisfied:
  \begin{enumerate}
    \item $f$ is a Morita equivalence of sketches;
    \item $f_!:\Mod^{K_\sk C}(\sk C)\rightarrow\Mod^{K_\sk D}(\sk D)$ is a Morita equivalence of sketches.
  \end{enumerate}
  The equivalence of these conditions follows from transitivity, Proposition \ref{constructible/transitive}.
  
  It is \emph{constructibly essential} if it is essential and $f^\dagger$ preserves constructible objects, i.e.~it extends to an adjunction $\Mod^{K_\sk C}(\sk C)\rightleftarrows\Mod^{K_\sk D}(\sk D)$.
  
\end{definition}

\begin{corollary}[Constructible Morita equivalence from Morita equivalence] \label{constructible/morita/criterion}

  Let $f:\sk C\rightarrow \sk D$ be a Morita equivalence, $K_\sk C$, $K_\sk D$ some sets of constructions such that $f(\sk C)\subseteq \sk D$.
  Suppose that each element of $K_\sk D$ admits an $f$-rectification whose underlying $\sk C$-diagram is $K_\sk C$-constructible.
  Then $f:(\sk C\sep K_\sk C)\rightarrow (\sk D\sep K_\sk D)$ is a constructible Morita equivalence.
  
\end{corollary}
\begin{proof}

  Since $f$ is a Morita equivalence and $h$ is $R(\sk D)$-acyclic, $(I\sep U)\rightarrow (C\downarrow_DJ\sep\source)$ is an $R(\sk C)$-equivalence; hence, the latter is $K_\sk C$-constructible.
  We may now apply Proposition \ref{constructible/essential/criterion}. \qedhere
  
\end{proof}

\begin{corollary}[Colimits of constructible modules] \label{constructible/modules/idempotence}

  Let $(\sk C\sep K_\sk C)$ be a sketch with constructions.
  Then $\Mod^K(\sk C)$ admits colimits of any diagram admitting a $\Yoneda_\sk C$-rectification.
  If $K'\subseteq \Cat\slice \Mod^K(\sk C)$ is the set of all such diagrams, then the extension $\Mod^K(\sk C)\cup K'-\mathtt{colimits}$ is Morita-trivial.
  
\end{corollary}

We end with a criterion in the same vein for constructing adjunctions between categories of constructible modules.
However, I do not know any examples in which this criterion is easy to check.

\begin{proposition}[Constructible adjunction] \label{constructible/essential/criterion}

  Let $f:(\sk C\sep K_\sk C)\rightarrow (\sk D\sep K_\sk D)$ be a morphism of sketches with constructions which is essential as a morphism of sketches.  
  Suppose that for each $(J\sep V)\in K_\sk D$, the diagram $C\downarrow_DJ\rightarrow C$ is $K_\sk C$-constructible.
  Then the right adjoint to $f$ preserves constructible objects.

\end{proposition}
\begin{proof}

  By Proposition \ref{shape/pullback}-\ref{shape/pullback/lax}, $f^\dagger|V|_D =|C\downarrow_DJ\rightarrow C|_\sk C$ which is by hypothesis $K_\sk C$-constructible whenever $V$ is $K_\sk D$-constructible. \qedhere

\end{proof}

\begin{example}[Constructible equivalences are not Morita equivalences]

  The notion of constructible equivalence is strictly weaker than that of constructible Morita equivalence, because the set of constructions $K$ may not be large enough to `see' all the relations $R$.
  For an example, take any sketch on $\point$ whose relations are connected spaces (cf.~\S\ref{sketch/on-point/}, and $K$ the set of diagrams indexed by sets.
  Regardless of $R$, we have $\Mod^\mathtt{finset}(\point\sep R) = \Set$.

\end{example}

\begin{para}[Constructible equivalence from orthogonality]

  There are interesting constructible equivalences that are not Morita equivalences.
  Let $(C\sep R)$ be a sketch and let $K$ be the set of diagrams $U:I\rightarrow C$ satisfying the following, equivalent, conditions:
  \begin{enumerate}
  
    \item 
      $|U|_C:\P(C)$ is $R$-continuous;
      
    \item 
      for all elements $V:J^\triangleright\rightarrow C$ of $R(\sk C)$ we have
      \[
        \colim_{i:I} C(V_\omega, U_i) \cong \lim_{j:J}\colim_{i:I} C(V_j, U_i) 
      \]
      in $\Space$ (i.e.~``the colimit over $(I\sep U)$ commutes with $R$-limits'').
      
  \end{enumerate}
  Then using condition 1, $\P^K(C)=\P_R(C)$ as a subset of $\P(C)$, i.e. $(C\sep \emptyset\sep K) \rightarrow (C\sep R\sep \Universe)$ is a constructible Morita equivalence.
  
  For example, if $C$ is $\kappa$-cocomplete and $\sk C=C_\kappa$ is the sketch of all $\kappa$-small colimits, then $K$ is the set of $\kappa$-filtered (small) diagrams in $C$, and we recover the familiar equation $\Ind_\kappa(C)=\P_\kappa(C)$.
  
  By analogy with this example, we might call a diagram satisfying the above conditions `$R(\sk C)$-filtered'.
  This should be compared to the concept of fibrancy for presheaves discussed in \ref{module/factorization-system}, which we also compared to `filteredness': certainly, the above criteria imply that the realisation is fibrant, but the converse does not always hold.
  
  This is an instance of the general observation that `filtered' conditions expressed in terms of a lifting condition tend to be weaker than corresponding conditions of limits commuting with colimits.
  The distinction between these two concepts is discussed in a the context of classical category theory in \cite{adamek2002classification}.
  
\end{para}


\subsection{Universal constructions}
\label{construction/universal/}

A set of categories defines a set of constructions for any sketch.
Given such a \emph{universal set of constructions} $K$, it is natural to ask when a category of $K$-constructible modules $\Mod^K(\sk C)$ itself admits colimits indexed by members of $K$.
In other words, when is the $K$-cocompletion functor idempotent? 
When does it agree with the cocompletion under \emph{iterated} $K$-colimits, as considered in \cite[\S5.3.6]{HTT} and \cite[\S4.4]{hinich2020yoneda}?

For many standard examples of $K$, this is satisfied: filtered categories, sifted categories, finite sets, $\kappa$-compact categories, the walking idempotent, $n$-connected diagrams, and higher Deligne-Mumford stacks.
In this section, we gather arguments establishing idempotence for these cases.
We find that in some cases it is natural to use rectifications as in Corollary \ref{constructible/modules/idempotence}, while in others is is simpler to use ad hoc or more sophisticated methods (i.e.~compactness).

References to definitions, theorems, and other environments in this subsection quote the number of the containing (sub)subsection.

\begin{definition*}[Universal sets of diagrams] \label{construction/universal/definition}

  A \emph{universal set of diagrams} is a set of categories, considered as a fully faithful Cartesian section $(-,K):\Sketch\rightarrow\Sketch^K$ of the forgetful functor $\forget^K:\Sketch^K\rightarrow\Sketch$.
  By abuse of notation, if $K\subseteq \Cat$ is a universal set of diagrams then for any sketch $\sk C$ we use the same letter to denote the set of diagrams in $\sk C$ having index category in $K$, or $K(\sk C)$ if ambiguity might thereby result.
  
  A universal set of diagrams induces a pointed endofunctor $\sk C\mapsto \Mod^K(\sk C)$ on $\Sketch$ with a unit map $\Yoneda_K:\id_\Sketch\rightarrow \Mod^K$ (but which is not, in general, a monad).
  A universal set of diagrams $K$ is \emph{idempotent} on a sketch $\sk C$ if $(\Mod^K(\sk C), K)$ is constructibly cocomplete, i.e.~if the $K$-cocompletion functor is idempotent on $\sk C$.
  
\end{definition*}

A sufficient general criterion for idempotence is the following:

\begin{proposition*}[Natural rectification] \label{construction/universal/criterion}

  Let $K\subseteq\Cat$ be a set of categories which is closed under Grothendieck integration, in that if $S\in K$ and $p:E\rightarrow S$ is a co-Cartesian fibration whose fibres belong to $K$, then $E\in K$.
  Let $\sk C$ be a sketch. Suppose that for each $F:\Mod(\sk C\sep K)$, $C\slice F\in K$.
  Then $\Mod(\sk C\sep K)$ is $K$-cocomplete.
  
\end{proposition*}
\begin{proof}

  Direct application of Corollary \ref{constructible/modules/idempotence}. \qedhere

\end{proof}

\begin{corollary*}[Rectification of strongly filtered diagrams] \label{construction/universal/criterion/filtered}

  Let $R\subseteq\Cat$ be a set of categories, and let $K\subseteq\Cat$ be the set of all categories $I$ such that colimits in $\Space$ over $I$ commute with $R$-limits.
  Then $K$-cocompletion is idempotent on all categories.

\end{corollary*}
\begin{proof}

  One easily verifies the hypotheses of Proposition \ref{construction/universal/criterion}. \qedhere 

\end{proof}

Unfortunately, by looking at the examples it is clear that this is not always easy to apply, especially when $\sk C$ has relations.

\subsubsection{Finite sets}

\label{construction/universal/finset}

  Let $\sk C:\Sketch$ and let $\mathtt{finset}$ be the set of finite sets.
  If $S$ is a finite set and $U:S\rightarrow \Mod^\mathtt{finset}(\sk C)$ is a map, then for each $s:S$ there is a finite set $T_s$ and a map $V_s:T_s\rightarrow C\slice U_s$ that exhibits $U_s$ as its coproduct in $\Mod(\sk C)$. 
  The square
  \[
    \begin{tikzcd}
      \coprod_{s:S}T_s \ar[d] \ar[r, "V"] & C \ar[d, "\Yoneda_\sk C"] \\
      S \ar[r, "U"'] & \Mod^\mathtt{finset}(\sk C)
    \end{tikzcd}
  \]
  is a rectification.
  Thus $\mathtt{finset}$-cocompletion is idempotent on any sketch.

\begin{remark*}[Finite sets via natural rectification]

  An approach to idempotence of $\mathtt{finset}$-cocompletion via Proposition \ref{construction/universal/criterion} might proceed by enlarging $\mathtt{finset}$ to the constructibly equivalent class $K\subseteq\Cat$ of categories that admit a cofinal functor from a finite set.
  But then, proving that $K$ is stable under Grothendieck integration requires us to repeat essentially the same argument, and even then it isn't clear how to apply Proposition \ref{construction/universal/criterion} to sketches with nontrivial relations.
  In general, the rectification constructed above depends on choices and cannot be made functorial for all sketches.
  
\end{remark*}

\subsubsection{Filtered categories} \label{constructible/example/filtered}

  Let $\kappa\mathtt{filt}$ be the set of small diagrams $I$ such that colimits over $I$ commute with $\kappa$-small limits in $\Space$.
  Then by Corollary \ref{construction/universal/criterion/filtered}, $\kappa\mathtt{filt}$-completion is idempotent on $\Cat$.
  
  Compared to the proof of \cite[Prop.~5.3.5.3]{HTT}, this argument avoids model category techniques and reduction to known facts about classical categories.
  On the other hand, we are using an \emph{a priori} stronger definition of filteredness than \emph{loc.~cit.}, and the proof \cite[Prop.~5.3.3.3]{HTT} that the two notions are equivalent goes by reduction to a statement about posets into which our techniques provide no new insight.

\subsubsection{Sifted categories} \label{constructible/example/sifted}

  Similar logic applies to a set $\mathtt{sift}$ of \emph{sifted} small categories \pathcite{htt/sifted/definition}.
  To see this, we need a minor reformulation of the definition of \emph{loc.~cit}:
  \begin{lemma*}
    
    A category $I$ is sifted if and only if the colimit functor $\Space^I\rightarrow\Space$ preserves finite products.
    
  \end{lemma*}
  \begin{proof}
  
    Suppose $\colim_I$ preserves finite products.
    Then for any $i,j:I$ the groupoid completion of $(i\times j)\downarrow I$ (product of representable functors in $\Space^I$) is a point, whence $I\rightarrow I\times I$ is cofinal.
    Moreover, $I$ is nonempty because 
    \[
      \colim_\emptyset:\Space^\emptyset\cong\point\rightarrow\Space, \quad * \mapsto \emptyset
    \]
    does not preserve the final object.
    
    The converse statement is \pathcite{htt/sifted/colimits-preserve-products}. \qedhere
    
  \end{proof}
  Again by Corollary \ref{construction/universal/criterion/filtered}, $\mathtt{sift}$-cocompletion is idempotent on $\Cat$.
  
   
  Compared to the proof of \cite[Prop.~5.5.8.10.(4)]{HTT}, this argument does not assume that $C$ admits finite coproducts. (I don't know any interesting examples that make use of this added generality, though).

\subsubsection{Idempotents} \label{construction/universal/idempotent}

  Let $\Idem$ denote the walking idempotent \pathcite{htt/idempotent/walking}, that is, the 1-category with one object $o$ whose endomorphism monoid is $\End(o)=\{1, e|e^2=e\}$.
  Thus $\{\Idem\}$-completeness is idempotent completeness.
  
  Let $\sk C$ be any subcanonical sketch, $X:\Mod^{\{\Idem\}}(\sk C)$ a retract of a representable, $e:\End(X)$ an idempotent.
  By constructibility, there is an object $X':C$, idempotent $\epsilon:\End_C(X')$, and retraction $s:X\rightleftarrows X':r$ in $\P(C)$ such that $sr=\Yoneda_{\sk C}(\epsilon)$.
  Then $ser:\End(\Yoneda_\sk C(X'))$ is an idempotent which by subcanonicity lifts to an idempotent $e':\End(X')$, whose splitting also lives in $\Mod^{\{\Idem\}}$.
  That is, $\{\Idem\}$-cocompletion is idempotent on subcanonical sketches.

  This question can also be approached using Proposition \ref{construction/universal/criterion}, though the argument is not exactly simpler.
  The relevant class $\mathtt{idem}$ of categories is those that admit a \emph{final idempotent}.
  A final idempotent on a category $I$ is a cone over the identity functor.
  That is, it consists of the following equivalent data (see \eqref{cone}):
  \begin{itemize}
    \item 
      A section of the projection $I\slice i\rightarrow I$.
      
    \item
      A natural transformation $\id_I\rightarrow \underline{i}$.
    
    \item
      A section of the terminal map $\Yoneda_I(i)\rightarrow\point_I$ in $\P(I)$. In particular, the final presheaf is a retract of a representable presheaf.
  \end{itemize}
  The walking idempotent $\Idem$ admits a unique final idempotent.
  A final idempotent with apex $i$ whose value at $i$ is the identity map exhibits $i$ as a final object of $I$.
  If $I$ admits a final idempotent at $i$ and $h:I\rightarrow J$ is cofinal, then the defining property of cofinality provides a section of $J\slice hi\rightarrow J$.
  To apply Proposition \ref{construction/universal/criterion}, we still need to show that $\mathtt{idem}$ is closed under $\int$. 
  \begin{lemma*}
  
    Suppose $I$ admits a final idempotent $(i,\sigma)$ and $p:J\rightarrow I$ is a co-Cartesian fibration such that $J_i$ admits a final idempotent $(j,\tau)$.
    Then $J$ admits a final idempotent at $j$.
    
  \end{lemma*}
  \begin{proof}
  
    Considered as a natural transformation $\id_J\rightarrow \underline{j}$, the desired final idempotent is given by the composite natural transformation
    
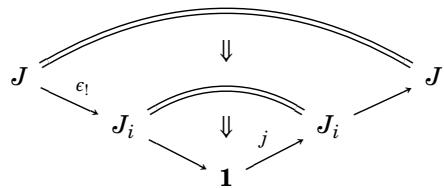
\[
      \begin{tikzcd}[row sep = tiny]
        && {} \ar[d, phantom, "\Downarrow" very near start] \\
        J \ar[rrrr, bend left, equals] 
        \ar[dr, "\epsilon_{!}"] && {} && J \\
        & J_i \ar[dr] \ar[rr, bend left, equals] \ar[rr, phantom, "\Downarrow"] & 
        & J_i \ar[ur] \\
        && \point \ar[ur, "j"]
      \end{tikzcd}
    \]
    with the upper cell $\id_J\rightarrow \epsilon_{!}$ provided by co-Cartesian pushforward along the final idempotent $\epsilon:\id_I\rightarrow\underline{i}$ of $I$. \qedhere
    
  \end{proof}
  If $\sk C$ is subcanonical and $F:\P(\sk C)$ is in the idempotent completion of $\sk C$, then a presentation of $F$ as the splitting of an idempotent $(X,p)$ is the same data as a final idempotent on $C\slice F$.
  In particular, if $C\slice F$ admits a final idempotent and $L:\P(C)\rightarrow\P(C)$ preserves representable functors, then $C\slice LF$ admits a final idempotent.
  Thus we may apply Proposition \ref{construction/universal/criterion} to rediscover that $\mathtt{idem}$-cocompletion is idempotent on subcanonical sketches.
  Compared to the simpler direct argument, this provides more information about diagrams whose colimits are representable in the idempotent completion.
  
\subsubsection{Categories of restricted shape}

There are interesting universal sets of diagrams defined at the level of the groupoid completion.
A \emph{connectivity class} is a set of spaces $A\subseteq\Space$ satisfying the following properties:
\begin{enumerate}
  \item
    $\point\in A$.
    
  \item
    $A$ is closed under (small) weakly contractible colimits in $\Space$.
    
  \item
    If $(X,x)\rightarrow (Y,y)\rightarrow (Z,z)$ is a fibre sequence of pointed spaces, and $X$ and $Z$ are in $A$, then $Y\in A$.
\end{enumerate}
A diagram $(I\sep U)$, resp.~presheaf $F$, on a category $C$ is said to be \emph{externally $A$-connected} if $I$, resp.~$C\slice F$, has groupoid completion in $A$.
The set of categories having groupoid completion in $A$ is denoted $A-\mathtt{conn}$.
The set $A-\mathtt{conn}$ is closed under cofinal maps of diagrams and, by condition 3) in the definition of a connectivity class, Grothendieck integration.

\begin{lemma*}

  Let $\sk C$ be a sketch whose relations have $A$-connected index categories, and let $F$ be an externally $A$-connected presheaf.
  Then $F^+$ is externally $A$-connected.
  
\end{lemma*}
\begin{proof}

  Recall \eqref{sketch/module/localizer/explicit} that $F^+$ may be expressed as the colimit of a diagram
  \[
    \target(f)\sqcup_{\source(f)}F: 
    \left( \int_{f\in \bar{R}} \langle f, [F\rightarrow\point]\rangle \right)^\triangleleft
    \rightarrow \P(C)
  \]
  whose index category is weakly contractible.
  Since each of $\target(f)$, $\source(f)$, $F$ is $A$-connected, the colimit is $A$-connected by closure under weakly contractible colimits. \qedhere
  
\end{proof}

\begin{corollary*}[Constructions defined by their shape] \label{construction/universal/shape}

  Let $A\subseteq\Space$ be a connectivity class, and let $A-\mathtt{conn}\subseteq\Cat$ be the set of categories whose groupoid completion belongs to $A$.
  Let $\sk C$ be a sketch whose relations have index categories that belong to $A-\mathtt{conn}$.
  Then $A-\mathtt{conn}$-cocompletion is idempotent on $\sk C$.
  
\end{corollary*}
\begin{proof}

  By the Lemma and Proposition \ref{construction/universal/criterion}. \qedhere

\end{proof}

\begin{example*}[$n$-connected categories] \label{construction/universal/connected}

  The set of $n$-connected spaces form a connectivity class for any $n=\N\sqcup\{\infty\}$.
  Indeed, this set is precisely the preimage of $\point$ under the colimit-preserving functor $\tau_n:\Space\rightarrow \Space_{\leq n}$, and the colimit of a constant functor with value $\point$ over a weakly contractible category is contractible.
  Corollary \ref{construction/universal/shape} applies to tell us that $n-\mathtt{conn}$-cocompletion is idempotent on any sketch with $n$-connected relations.
  
\end{example*}

\subsubsection{Finite categories}

 \label{constructible/example/compact}

  An $\infty$-category $I$ is \emph{finite} if it can be exhibited as a finite iterated pushout of finite coproducts of $n$-simplex categories $\Delta^n$.  
  Let $\mathtt{fin*}$ be the union of the set of finite categories and $\{\Idem\}$. 
  Let $\sk C$ be a sketch with finite relations.
  Then $\mathtt{fin*}$ is idempotent on $\sk C$: by \pathcite{htt/compact/generation/from-localization}, $\Mod(\sk C)$ is compactly generated and representable objects are compact.
  It follows from \emph{loc.~cit}.~and \pathcite{htt/compact/presheaf/criterion} that an object of $\Mod(\sk C)$ is compact if and only if it is a retract of a finite colimit of representable objects, that is, if it is in $\Mod^\mathtt{fin*}(\sk C)$.
  
  Similar arguments apply with `finite' replaced by `$\kappa$-small'.

\begin{remark*}

  It does not seem to be so easy to show that $\kappa$-small colimits are composable using `diagrammatic' means, i.e.~by constructing rectifications of $\kappa$-small diagrams of $\kappa$-compact modules and applying Corollary \ref{construction/universal/criterion}.
  
\end{remark*}

\begin{example*}[Non-rectification of finite diagrams]

  Let $C$ be the category of contractible manifolds of class $C^k$ with $k\in\{0,\ldots,\infty\}$.
  Then $\Sheaf(C)$ is not compactly generated, and so we cannot apply \eqref{constructible/example/compact}.
  Accordingly, the question of whether the set of finitely constructible objects of $\Sheaf(C)$ is closed under finite colimits is quite subtle.\footnote{I came to consider this question when studying the axioms of Spivak for derived manifolds, which guarantee closure under finite pushouts of open immersions \cite{spivak2010derived}.}
  
  For example:
  \begin{itemize}
  
    \item
      The universal cover $\pi:\R\rightarrow \R/\Z$ cannot be refined to a map of finite iterated pushout diagrams of contractible open immersions.
      However, $\R/\Z$ can be expressed as a colimit of a diagram $S^1\rightarrow \langle\R\rangle\subset C^k$, and $\pi$ does lift to this diagram.
      So $\pi$ is rectifiable by finite diagrams in $C^k$, but not by finite iterated pushouts.
      
    \item
      I was able to prove that for any \'etale map $f:\R^2\rightarrow S^2$, there exists a decomposition $\R^2=U\cup V$ into contractible open subsets such that the restriction of $f$ to $U$, resp.~$V$, avoids the north, resp.~south pole of $S^2$.
      In particular, $f$ is rectifiable by finite diagrams in $C^\infty$.
      
      Based on these two examples, we may well guess that any \'etale map between finitely constructible manifolds is rectifiable by finite diagrams, from which it would follow that finite colimits of finitely constructible manifolds are finitely constructible.
  
    \item
      On the other hand, space-filling curves provide examples of $C^0$ maps, say, $f:\R\rightarrow S^2$ such that $f^{-1}(p)\simeq \Z$ for all $p:S^2$, and these are quite likely \emph{not} rectifiable by finite diagrams in $C^0$.
      So the same conclusion doesn't hold for arbitrary finitely constructible objects of $\Sheaf(C^0)$.
      (But I couldn't prove this!)
      
  \end{itemize}  

\end{example*}

\subsubsection{Higher Deligne-Mumford stacks} \label{construction/universal/stacks}

The following example doesn't fit exactly into the framework of Definition \ref{construction/universal/definition}, but nonetheless, it belongs to a family of sketches with constructions closed under taking module categories.
Let $\Aff$ be the category of affine schemes, considered as a sketch with relations the set $\mathtt{\acute{e}t-cov}$ of conical diagrams whose image in $\P(\Aff)^{\Delta^1}$ is a covering sieve for the \'etale topology.
Its continuous presheaves are, by definition, sheaves (of spaces) for the \'etale topology: $\Mod(\Aff\sep\mathtt{\acute{e}t-cov})=\Sheaf(\Aff\sep\et)$.

Let $\Aff^\et\subset\Aff$ be the subcategory spanned by all objects and the \'etale morphisms between them, and let $K\subseteq\Cat\slice\Aff$ be the set of diagrams that factor through $\Aff^\et$.
Sheaves constructible for this set are called \emph{higher Deligne-Mumford stacks} \cite{carchedi2020higher}.
We write $\DM \defeq \Mod^K(\Aff,\mathtt{\acute{e}t-cov})$.
By \eqref{construction/slice} we have $\DM\slice X=\Mod^{K\slice X}(\Aff\slice X\sep \mathtt{\acute{e}t-cov})$ for any $X:\Sheaf(\Aff\sep\et)$, and by Lemma \ref{constructible/yoneda/properties} the inclusions
\[
  \Aff\slice X \subset \DM\slice X  \subset\Sheaf(\Aff\sep\et)\slice X
\]
are dense.

The subcategory $\Aff^\et$ itself has the structure of a Grothendieck site (hence sketch), and the inclusion $j:\Aff^\et\hookrightarrow \Aff$ preserves covering sieves (hence relations).
The essential image of $j_!:\Sheaf(\Aff^\et)\rightarrow\Sheaf(\Aff)$ is $\DM$.
In particular, for any $X\in \DM$ the diagram
\[
  [\Aff^\et\slice X \rightarrow \Sheaf(\Aff\sep \et)\slice X ] \quad \in\quad K-\mathtt{colimits}
\]
is a colimit, because $\Aff^\et\slice X\rightarrow\Sheaf(\Aff^\et\sep\et)\slice X$ is a colimit by density of the Yoneda embedding $\Aff^\et\rightarrow\Sheaf(\Aff^\et)$ and because $j_!$ preserves colimits.

Now let $V:J\rightarrow\DM$ be an \'etale diagram of Deligne-Mumford stacks, that is, a diagram that factors through $\Sheaf(\Aff^\et\sep\et)$.
The preceding argument shows that
\[
  \begin{tikzcd}
    \Aff^\et\downarrow J \ar[r, "\source"] \ar[d, "\target"] 
    \ar[dr, phantom, "\Downarrow" marking]
    & \Aff \ar[d, "\Yoneda_{\Aff,K}"] \\
    J \ar[r, "V"] & \DM
  \end{tikzcd}
\]
is a $\Yoneda_{\Aff,K}$-rectification in the sense of Definition \ref{rectification/definition}.
Hence by Corollary \ref{constructible/modules/idempotence}, $\DM\subseteq\Sheaf(\Aff\sep\et)$ is closed under \'etale colimits; or more idiomatically, it \emph{admits \'etale descent}.

\begin{remark*}

  Although we fixed attention on classical schemes for concreteness, this argument continues to hold true for any \emph{strong \'etale blossom} in the sense of \cite[Def.~5.1.7]{carchedi2020higher}.
  That is, $\mathtt{\acute{e}t-cov}$ can be considered as a universal set of diagrams in the 2-category of strong \'etale blooms (rather than $\Cat$), and the monad it induces is idempotent.
  Other examples of strong \'etale blooms can be constructed from $C^r$ manifolds ($r\in\N \sqcup\{\infty,\omega\}$), complex manifolds, non-Archimedean analytic spaces, and derived geometric generalisations thereof.
  
\end{remark*}

\printbibliography
\end{document}